\documentclass[12pt]{article}
\usepackage{subcaption}
\usepackage[font=scriptsize]{caption}

\usepackage{latexsym,amsmath,amsthm,amssymb,amsfonts,amscd,multirow}
\usepackage{epsfig,verbatim,epstopdf,graphics}
\usepackage{color}
\usepackage{array}
\usepackage{hyperref}
\usepackage[ruled,linesnumbered]{algorithm2e}
\usepackage{mathtools} 
\usepackage{lscape}

\newif\ifshowproof
 \showprooftrue

\newtheorem{Example}{Example}[section]
\newtheorem{Thm}{Theorem}[section]

\newtheorem{Lem}[Thm]{Lemma}
\newtheorem{rem}[Thm]{Remark}
\newtheorem{definition}{Definition}[section]

\newcommand{\ot}{\frac{1}{2}}

\newcommand{\ohj}{\frac{1}{h_j}}

\newcommand{\intj}{\int_{I_j}}

\newcommand{\uhat}{\hat u_h}

\newcommand{\vh}{v_h}
\newcommand{\uxt}{\widetilde{(u_h)_x}}
\newcommand{\jp}{{j+\frac{1}{2}}}
\newcommand{\jm}{{j-\frac{1}{2}}}
\newcommand{\uxa}{\{ (u_h)_x \}}
\newcommand{\ua}{\{ u_h \}}

\newcommand{\sumj}{\sum_{j=1}^N}
\newcommand{\suml}{\sum_{l=0}^{N-1}}

\newcommand{\ao}{\alpha_1}

\newcommand{\bo}{\beta_1}
\newcommand{\bt}{\beta_2}

\newcommand{\lo}{\lambda_1}
\newcommand{\lt}{\lambda_2}

\newcommand{\sumn}{\sum_{n=-\infty}^{\infty}}

\newcommand{\zn}{\mathbb{Z}_N}

\newcommand{\sch}{Schr\"odinger }
\newcommand{\beq}{\begin{equation}}
\newcommand{\eeq}{\end{equation}}
\newcommand{\beqn}{\begin{equation*}}
\newcommand{\eeqn}{\end{equation*}}
\newcommand{\beqa}{\begin{eqnarray}}
\newcommand{\eeqa}{\end{eqnarray}}
\newcommand{\bmat}{\begin{bmatrix}}
\newcommand{\emat}{\end{bmatrix}}

\newcommand{\pst}{P^\star_h}
\newcommand{\plt}{P_h^0}
\newcommand{\pdag}{P^\dagger_h}

\newcommand{\abs}[1]{\left | #1 \right |}
\newcommand{\la}{\Lambda}
\newcommand{\g}{\Gamma}
\newcommand{\wnorm}[3]{\|#1\|_{W^{#2,#3}(\mathcal I_N)}}
\newcommand{\wnormj}[3]{\|#1\|_{W^{#2,#3}( I_j)}}
\newcommand{\hnorm}[2]{\|#1\|_{H^{#2}(\mathcal I_N)}}
\newcommand{\hnormj}[2]{\|#1\|_{H^{#2}( I_j)}}
\newcommand{\bdnorm}[1]{\|#1\|_{ L^2 (\partial \mathcal I_N)}}
\newcommand{\bdnormj}[1]{\|#1\|_{L^2 (  \partial I_j)}}

\newcommand{\bs}{\boldsymbol}

\newcommand{\sumla}{\boxtimes}
\newcommand{\sumnj}{\boxplus}

\newcommand{\cb}{\color{blue}}
\DeclarePairedDelimiter{\floor}{\lfloor}{\rfloor}

\newcolumntype{H}{>{\setbox0=\hbox\bgroup}c<{\egroup}@{}}
\title
{Superconvergence of   ultra-weak discontinuous Galerkin methods for the linear \sch equation  in one dimension}

\author{Anqi Chen
\thanks{Department of Mathematics, Michigan State University,
East Lansing, MI 48824 U.S.A.
{\tt chenanq3@msu.edu}.}%
\and
 Yingda Cheng
\thanks{Department of Mathematics, Department of  Computational Mathematics, Science and Engineering, Michigan State University,
               East Lansing, MI 48824, USA.
E-mail: ycheng@msu.edu. Research is supported by NSF grants  DMS-1453661 and DMS-1720023.}
\and
Yong Liu
\thanks{School of Mathematical Sciences, University of Science and Technology of China, Hefei, Anhui 230026 People's Republic of China. {\tt yong123@mail.ustc.edu.cn.}
}
\and
Mengping Zhang
\thanks{School of Mathematical Sciences, University of Science and Technology of China, Hefei, Anhui, 230026 People's Republic of China. {\tt mpzhang@ustc.edu.cn.} Research supported by NSFC grant 11871448.}
}

\date{\today}

\begin{document}
\maketitle

\begin{abstract}
We analyze the superconvergence properties of   ultra-weak discontinuous Galerkin (UWDG) methods with various choices of flux parameters for one-dimensional linear \sch equation.  In our previous work \cite{2018arXiv180105875C}, stability and optimal convergence rate are established for a large class of   flux parameters.
Depending on the flux choices and if the polynomial degree $k$ is even or odd, in this paper, we prove $2k$ or $(2k-1)$-th order superconvergence rate for cell averages and numerical flux of the function, as well as $(2k-1)$ or $(2k-2)$-th order for numerical flux of the derivative. In addition, we prove superconvergence of $(k+2)$ or $(k+3)$-th order of the DG solution towards a special projection.  At a class of special points, the function values and the first and second order derivatives of the DG solution are superconvergent with order  $k+2, k+1, k$, respectively.  The proof relies on the correction function techniques initiated in \cite{cao2014superconvergence}, and applied to \cite{cao2017superconvergenceddg} for direct DG (DDG) methods for diffusion problems. 
Compared with \cite{cao2017superconvergenceddg}, \sch equation poses unique challenges for superconvergence proof because of the lack of the dissipation mechanism from the equation. One major highlight of our proof is that   we introduce specially chosen  test functions in the error equation and show the superconvergence of the second derivative and jump across the cell interfaces of the difference between numerical solution and projected exact solution. This technique was originally proposed in  \cite{cheng2010superconvergence} and is essential to elevate the convergence order for our analysis.  Finally, by negative norm estimates, we   apply the post-processing technique and show that the accuracy of our scheme can be enhanced to order $2k.$ Theoretical results are verified by numerical experiments.

\end{abstract}


\section{Introduction}

Discontinuous Galerkin (DG) methods belong to a class of finite element methods using   discontinuous piecewise function space for test functions and numerical solution. The first DG method was introduced by Reed and Hill in \cite{reed1973triangular} for solving neutron transport problems. A major development of DG methods is the Runge-Kutta DG (RKDG) framework introduced for solving hyperbolic conservation laws in a series of papers, see \cite{cockburn2001runge} for a review. Because of the completely discontinuous basis, DG methods have several attractive properties. It can be used on many types of meshes, even those with hanging nodes. The methods can be designed with $h$-$p$ adaptivity and very high parallel efficiency. 

We are interested in solving the following linear \sch equations by DG methods. 
\beq
\label{eqn:ls}
	\begin{aligned}
	&iu_t + u_{xx} = 0,		&&	(x,t) \in I\times (0,T_e]	,\\
	&u(x,0) = u_0(x),		&&	 
	\end{aligned}
\eeq
where $I=[a,b]$ and periodic boundary condition. Various types of DG schemes for discretizing the second order spatial derivatives have been used to compute \eqref{eqn:ls}, including the local DG
(LDG) method   \cite{xu2005local, MR2888305, XingETDLDGsch} and the direct DG (DDG) methods \cite{luddgsch}.  This paper will focus on  the ultra-weak DG (UWDG) methods,  which can be traced backed to \cite{cessenat1998application}, and  refer to those DG methods \cite{shu2016discontinuous} that rely on repeatedly applying integration by parts  so all the spatial derivatives are shifted from the solution to the test function in the weak formulations \cite{MR2373176, XingKdvDG}. In our previous work \cite{2018arXiv180105875C}, UWDG methods were studied, and a systematic choice of flux parameters were made to guarantee stability or energy conservation property of the scheme.
Moreover, using projection techniques,   convergence results of the UWDG method for the one-dimensional nonlinear \sch equation were established. It was shown that a wide range of flux parameter choices can yield optimally convergent scheme.
In this work, we continue the research and investigate   superconvergence of the UWDG scheme.

The study of superconvergence  is of importance because \emph{a posteriori} error estimates can be derived guiding adaptive calculations.   For    superconvergence of DG methods, many results exist in the literature. We refer the readers to \cite{adjerid2006superconvergence,adjerid2002posteriori} for ordinary differential equation results.
In \cite{cheng2010superconvergence}, Cheng and Shu proved that the DG and LDG solutions are $(k+3/2)$-th order superconvergent towards projections of exact solutions of hyperbolic conservation laws and convection-diffusion equations using specially designed test functions when piecewise polynomials of degree $k$ are used. For linear hyperbolic problems, in \cite{yang2012analysis}, Yang and Shu proved that, under suitable initial discretization, the DG solutions of linear hyperbolic systems are convergent with optimal $(k+2)$-th order at Radau points. More recently, in \cite{cao2014superconvergence}, Cao \emph{et al} proved the $(2k+1)$-th superconvergence rate for cell average and DG numerical fluxes by introducing a locally defined correction function. The correction function also helps simplify the proof for point wise $(k+2)$-th superconvergence rate at Radau points and prove the derivative of DG solution has $(k+1)$-th superconvergence rate at so-called ``left Radau'' points. Then this technique has been extended to prove the superconvergence of DG solutions for linear and nonlinear hyperbolic PDEs in \cite{cao2018superconvergence,cao2017superconvergence}, DDG method for convection diffusion equations \cite{cao2017superconvergenceddg} and LDG method for linear \sch equations  \cite{zhou2017superconvergence}. Overall, for equations with higher order spatial derivatives, the same type of correction functions can be used for the LDG method which is based on a reformulation into a first order system of equations. For DDG method, new correction functions are needed treating the second order derivative directly \cite{cao2017superconvergenceddg}.

Another type of superconvergence of DG methods is achieved by postprocessing the solution by convolution with a kernel function, which is a linear combination of B-spline functions. For linear hyperbolic systems, \cite{cockburn2003enhanced} provided a framework for constructing such postprocessor and proving the superconvergence of the postprocessed DG solutions. Through the analysis of negative norm estimates and divided difference estimates, they showed that the postprocessed solution is superconvergent at a rate of $2k+1$. More recently, in \cite{Ji2013, Meng2017} the analysis are extended to scalar nonlinear hyperbolic equations. 

In this work,   we aim at the study of superconvergence of the UWDG methods for \eqref{eqn:ls} with scale invariant flux parameters. Such choice include  all commonly used fluxes, e.g. alternating, central, DDG and interior penalty DG (IPDG) fluxes. Depending on the flux choices and the evenness of oddness of the polynomial degree $k$,  we obtain $2k$ or $(2k-1)$-th order superconvergence rate for cell averages and numerical flux of the function, as well as $(2k-1)$ or $(2k-2)$-th order for numerical flux of derivative. 
The proof relies on the correction function techniques for second order derivatives applied to \cite{cao2017superconvergenceddg} for DDG methods for diffusion problems. This correction function  also enable us to prove the UWDG solution is superconvergent with a rate of $k+3$ to the special projection $\pst$ we introduced in \cite{2018arXiv180105875C} if $k \ge 3.$ We show that the function values and the first and second order derivatives of the DG solution are superconvergent with order  $k+2, k+1, k$, respectively, at interior points whose locations are determined by roots of certain polynomials associated with the flux parameters.
We want to emphasize that our approach is related but different from the superconvergence proof in \cite{cao2017superconvergenceddg} for diffusion equations, mainly because there is no dissipation mechanism in the \sch equation. Therefore, when $k$ is even, there is some additional terms in the error estimates that cannot be bounded. To overcome this difficulty,  we take specially chosen  test functions in the error equation and show the superconvergence of some intermediate quantities. This technique was originally proposed in  \cite{cheng2010superconvergence} and is essential to elevate the convergence order for our scheme. For the postprocessed UWDG solution, we introduce a dual problem and prove $(2k)$-th order negative norm estimate. The order is also one order less than that in hyperbolic equations, and again, due to the ultra-weak formulation which has boundary term of the product of derivatives and function values. With the negative norm estimates and divided difference estimates, we prove the $(2k)$-th order superconvergence rate for the postprocessed solution.  


  The rest of the paper is organized as follows. In Section \ref{sec:scheme}, we recall the UWDG scheme for linear \sch equations and some properties of the spatial discretization. In Section \ref{sec:proj}, we  define notations and projections. Section \ref{sec:suplinear} contains the main results of the paper,  superconvergence of the UWDG solution in various measures. In Section \ref{sec:numerical}, we provide numerical tests verifying theoretical results. Finally, we conclude in Section \ref{sec:conclusion}. Some technical proof is provided in the Appendix.

\section{Numerical Scheme}
\label{sec:scheme}

We first define notations of the mesh and finite element solution space. For interval $I=[a,b]$, the usual DG mesh $\mathcal I_N$ and the index set  $\mathbb{Z}_N = \{ 1, 2, \cdots, N \}$ is defined as:
\[
a=x_{\ot} < x_{\frac{3}{2}} < \cdots < x_{N+\ot} =b,
\quad I_j=(x_{\jm},x_{\jp}), \, x_j=\ot(x_{\jm}+x_{\jp}), \, j \in   \mathbb{Z}_N
\]
and
\[
h_j=x_{\jp}-x_{\jm}, \quad h=\max_{j \in   \mathbb{Z}_N} h_j,
\]
with mesh regularity requirement $\frac{h}{\min h_j} < \sigma$, $\sigma$ is fixed during mesh refinement.
The approximation space, which is a piecewise \emph{complex} polynomial space on $\mathcal I_N$ is defined as:
\[
V_h^k=\{v_h : v_h|_{I_j} \in P_c^k(I_j),\, I_j \in \mathcal I_N  \},
\]
where $P_c^k(I_j)$ is the space of \emph{complex} polynomials of degree up to $k$ on cell $I_j$. For a function
$v_h \in V_h^k$, we use $(v_h)^-_{\jm}$ and $(v_h)^+_{\jm}$ to denote the value of $v_h$ at $x_{\jm}$
from the left cell $I_{j-1}$ and the right cell $I_j$ respectively. The jump and average values are defined as $[v_h]=v_h^+-v_h^-$ and $\{v_h\}=\ot(v_h^++ v_h^-)$ at cell interfaces.
 
Throughout the paper, we use the standard Sobolev norm notations $\| \cdot \|_{W^{s,p}(I)}$ and broken Sobolev space on mesh $\mathcal I_N$. We denote   $\hnorm{v}{s}^2 = \sumj \hnormj{v}{s}^2$ and $\wnorm{v}{s}{\infty} = \max_j \wnormj{v}{s}{\infty}$. In Section \ref{sec:postp}, we consider negative norms and the definition is $\| v \|_{H^{-l} (I)} = \sup_{\Phi \in \mathcal{C}_0^\infty(I)} \frac{\int_I v(x) \Phi(x) dx}{\| \Phi\|_{{H^l} (I)}}.$   Additionally, we denote by $\bdnorm{v}$ the broken $L^2$ norm on cell interfaces, i.e., $\bdnorm{v}^2 = \sumj \bdnormj{v}^2$, where$\bdnormj{v}^2 = (v^-_{x_{j+\ot}})^2 + (v^+_{x_{j-\ot}})^2$.  We also denote $\|\cdot\|=\|\cdot\|_{L^2(I)}=\|\cdot\|_{L^2(\mathcal I_N)}$  to shorten the notation.  Lastly, we recall inverse inequalities
\beq
\label{eqn:inveq}
	\begin{aligned}
\| (v_h)_x \|_{L^2(\mathcal{I}_j)}  \leq  C h_j^{-1} \| v_h \|_{L^2(\mathcal{I}_j)}, \quad \| v_h \|_{L^2(\partial \mathcal I_j)} \leq C h^{-\ot} \| v_h \|_{L^2(\mathcal{I}_j)},  \\
\quad \| v_h \|_{L^\infty(\mathcal{I}_j)} \leq C h^{-\ot} \| v_h \|_{L^2(\mathcal{I}_j)},\ \forall v_h \in V_h^k,
\end{aligned}
\eeq
and trace inequalities
\beq
\label{eqn:traceeq}
\| v \|^2_{L^2(\partial I_j)} \leq C h_j^{-1} \| v \|^2_{L^2(I_j)},
\eeq
where, here and below $C$ is a constant independent of  the function $u$ and the mesh size $h$.

The semi-discrete UWDG scheme formulated in  \cite{2018arXiv180105875C} is defined as follows: we
solve for the unique function $u_h=u_h(t) \in V_h^k, \, k \ge 1, \, t\in (0,T_e]$, such that
\beq
\label{eqn:scheme}
\begin{aligned}
a_j(u_h,v_h)=0, \quad \forall j \in   \mathbb{Z}_N
\end{aligned}
\eeq
holds for all $\vh \in V_h^k,$  where 
\begin{equation*}
a_j(u_h,v_h) = \intj (u_h)_t \vh dx - iA_j(u_h,v_h),	\\
\end{equation*}
with $A_j(u_h,v_h) = \intj u_h (\vh)_{xx} dx - \uhat (\vh)^-_x |_\jp + \uhat (\vh)^+_x |_\jm  
+ \uxt \vh^- |_\jp -\uxt \vh^+ |_\jm$ as the UWDG spatial discretization for the second order derivative term.
The ``hat" and``tilde" terms are the numerical fluxes   for $u$ and $u_x$ at cell boundaries, which 
are single valued functions defined as: 
\beq
\label{eqn:flux}
	\begin{aligned}
	\uhat	&=\ua -\ao [u_h] +\bt[(u_h)_x],		& \ao,\bt \in \mathbb{R}.	 \\
        \uxt	& =\uxa+\ao [(u_h)_x] +\bo[u_h],	&  \bo \in \mathbb{R},	\\
	\end{aligned}
\eeq
where $\ao,   \bo, \bt   $ are prescribed parameters that may have $h$ dependence.
Note that we can rewrite the flux definition above in a matrix form
\beq
\label{eqn:fluxmat}
\bmat
\uhat		\\ 	\uxt
\emat
= 
G
\bmat
u_h^-		\\	(u_h)^-_x
\emat
+
H
\bmat
u_h^+		\\	(u_h)_x^+
\emat,
\quad 
 G
 = \bmat
\ot+\ao & -\bt \\
-\bo	& \ot - \ao
\emat
, \ 
H
=I_2-G=\bmat
\ot-\ao & \bt \\
\bo	& \ot + \ao
\emat
,
\eeq
where $I_2$ denotes the $2\times2$ identity matrix.
Some commonly used fluxes take the following choices of parameters.
\begin{itemize}
\item central flux, $\ao = \bo = \bt = 0;$
\item alternating flux, $\ao  = \pm \ot, \bo = \bt = 0;$
\item IPDG like flux, $\ao = \bt = 0,\bo = \tilde \bo h^{-1};$
\item DDG like flux, $\ao = \tilde \ao,  \bt = 0,\bo = \tilde \bo h^{-1};$
\item more generally, any scale invariant flux, $\ao =\tilde \ao, \bo = \tilde \bo h^{-1}, \bt = \tilde \bt h;$
\end{itemize}
where $\tilde \ao, \tilde \bo, \tilde \bt$ are prescribed constants independent of mesh size. For simplicity, in this paper we will only consider scale invariant flux choices. We now introduce 
\begin{align*}
&a(u_h, v_h ) = \sum_{j=1}^N a_j(u_h,v_h),\\
&A(u_h,v_h) = \sumj A_j(u_h,v_h)= \int_I u_h (v_h)_{xx} dx + \sumj  \left ( \hat u_h [(v_h)_x] - \widetilde{(u_h)_x}[v_h] \right ) \big \rvert_{\jp}.
\end{align*}
Clearly, the scheme boils down to requiring $a(u_h, v_h )=0, \forall \vh \in V_h^k.$

\medskip

 The following lemma shows  the symmetry property of $A(\cdot, \cdot)$.
\begin{Lem}[Symmetry of $A(\cdot, \cdot)$]
For $u, v \in H^{2}(\mathcal I_N)$ satisfying periodic boundary condition, we have $A(u,v) = A(v,u).$ Furthermore, $A(v, \bar v) \in \mathbb{R}$. Here and in what follows, the overline means complex conjugate.   
\end{Lem}

\begin{proof}
%
From integration by parts, we have
\begin{align*}
A(u,v)  = - \int_I u_x v_{x} dx + \sumj  \left ( \hat u [v_x] - [uv_x] - \widetilde{u_x}[v] \right ) \big \rvert_{\jp}.
\end{align*}
Similarly, 
$
A(v,u)= - \int_I u_x v_{x} dx + \sumj  \left ( \hat v [u_x] - [vu_x] - \widetilde{v_x}[u] \right ) \big \rvert_{\jp}.
$
Plugging in the definition of the numerical fluxes in \eqref{eqn:flux}, we have at $x_{j + \ot}, \forall j \in \zn$ 
\begin{align*}
 \hat u [v_x] - [uv_x] - \widetilde{u_x}[v] & = \big (\{u\} - \ao [u] + \bt [u_x] \big ) [v_x] - (\{u\}[v_x] + [u]\{v_x\}) \\
 			& - \big (\{u_x\} + \ao [u_x] + \bo [u] \big )[v]		\\
			& = [u_x] \big (\{v\} -\ao [v] + \bt [v_x] \big)	- [u] \big ( \{v_x\} +\ao [v_x] + \bo [v] \big)															\\
			& - \big ( [u_x] \{v\} + \{u_x\} [v] \big)			\\
			& = [u_x]\hat v - [u] \widetilde{v_x} - [u_x v],
\end{align*}
then the proof for $A(u,v) = A(v,u)$ is complete. It follows that $A(v, \bar v) - \overline{A(v, \bar v)} = A(v, \bar v) - A(\bar v, v) = 0,$ which implies $A(v, \bar v) \in \mathbb{R}.$\end{proof}

In our previous work \cite{2018arXiv180105875C}, we proved that our semi-discrete scheme is energy conservative, which is a direct result of the lemma above:
\beq
\label{eqn:l2stab}
0=a(u_h, \overline{u_h}) + \overline{a(u_h, \overline{u_h})} = \frac{d}{dt} \| u_h \|^2 - iA(u_h, \bar u_h) + i A(u_h, \bar u_h) =  \frac{d}{dt} \| u_h \|^2.\eeq
This property of our scheme is consistent with the energy conservation property of \sch equations. It is essential to have a symmetric $A(u_h,v_h)$ for designing a finite element scheme which is energy-preserving for \sch equations. Compared with discretization for diffusion equations, we don't have any extra diffusion term in \eqref{eqn:l2stab} to help with the estimates.  Therefore,  superconvergence error estimates are more challenging compared with \cite{cao2017superconvergenceddg}.

\section{Notations and Projections}
\label{sec:proj}
To facilitate the discussion, we introduce notations and define projection operators that will be used    in the paper.

\subsection{Notations}

We introduce some notations first. We define the Legendre expansion of a function $u \in L^2(I)$ on cell $I_j$ as follows,
\beq
\label{eqn:lt}
u|_{I_j} = \sum_{m=0}^\infty u_{j,m} L_{j,m}(x),
\eeq
where $L_{j,m}(x) : = L_m (\xi), \xi = \frac{x-x_j}{h_j/2}$, and $L_m(\cdot)$ is the standard Legendre polynomial of  degree $m$ on $[-1,1]$. In what follows, we write $L_{j,m}(x)$ as $L_{j,m}$, and $L_m(\xi)$ as $L_m$ for notational convenience.
We can compute $u_{j,m}$ using orthogonality of Legendre polynomials and Rodrigues' formula,
\beq
\label{eqn:ltcoefcompute}
\begin{aligned}
u_{j,m} 	& = \frac{2m+1}{h_j} \intj u(x) L_{j,m} dx  = \frac{2m+1}{2} \int_{-1}^1 \hat u_j (\xi) L_m d\xi	  = \frac{2m+1}{2} \frac{1}{2^m m!} \int_{-1}^1 \hat u_j (\xi) \frac{d}{d\xi^m} (\xi^2 -1)^m d\xi \\
		& = \frac{2m+1}{2} \frac{(-1)^l}{2^m m!} \int_{-1}^1 \frac{d}{d\xi^l} \hat u_j (\xi) \frac{d}{d\xi^{m-l}} (\xi^2 -1)^m d\xi,
\end{aligned}
\eeq
where $\hat u_j(\xi)=u(x(\xi))$ is defined as the function $u|_{I_j}$ transformed to the reference domain $[-1,1]$. By Holder's inequality, if $u \in W^{l,p}(I)$,
\beq
\label{eqn:ltcoef}
\abs{u_{j,m}} \leq C h_j^{l-\frac{1}{p}} |u|_{W^{l,p}(I_j)}, \quad 0 \leq l \leq m.
\eeq


Similar to \cite{cao2014superconvergence}, we define operator $D^{-1}$ for any integrable function $v$ on $I_j$  by
\beq
\label{eqn:ddef}
D^{-1} v(x) = \frac{2}{h_j} \int_{x_{\jm}}^x v(x) dx = \int_{-1}^\xi \hat v(\xi) d\xi, \quad x \in I_j.
\eeq
Using the property of Legendre polynomials, we have
\begin{subequations}
\begin{align}
D^{-1} L_{j,k} & = \frac{1}{2k+1} \left ( L_{j,k+1} - L_{j,k-1} \right ), \quad k \ge 1. 	\label{eqn:d-1}\\
D^{-2} L_{j,k} & = \frac{1}{2k+1} \left ( \frac{1}{2k+3} (L_{j,k+2} - L_{j,k}) - \frac{1}{2k-1} (L_{j,k} - L_{j,k-2}) \right ), \quad k \geq 2,	\label{eqn:d-2}
\end{align}
\end{subequations}
where $D^{-2} = D^{-1} \circ D^{-1}$.

Finally, we collect some additional notations that will be frequently used in the paper in Table \ref{tab:notation}.  

\begin{table}
\caption{Notations  for some frequently used quantities. Subscript $j$ can be dropped for uniform mesh.  }  
\label{tab:notation}
	\centering
	\begin{tabular}{c|c|c|c}
	\hline
	\textrm{Notation} & \textrm{Definition} &	\textrm{Notation} & \textrm{Definition} \\ \hline
	$G$ & 
$\bmat
\ot+\ao & -\bt \\
-\bo	& \ot - \ao
\emat$					&
$H$ &
$\bmat
\ot-\ao & \bt \\
\bo	& \ot + \ao
\emat$					\\\hline
	$\g_j$		& $\bo + \frac{\bt}{h_j^2}k^2(k^2-1) - \frac{2k^2}{h_j}(\ao^2 + \bo \bt + \frac{1}{4})$ &
	$\la_j$		& $ - \frac{2k}{h_j}(\ao^2 + \bo \bt - \frac{1}{4})$		\\\hline
	$L_{j,m}^-$ &
$\bmat
L_{j,m}(x_{\jp})	\\
\frac{2}{h_j} \frac{d}{dx}  L_{j,m}(x_{\jp})
\emat	$				&
	$L_{j,m}^+$ &
$\bmat
L_{j,m}(x_{\jm})	\\
\frac{2}{h_j} \frac{d}{dx} L_{j,m}(x_{\jm})
\emat$
						\\\hline
	$A_j$ & $G [L_{j,k-1}^-, L_{j,k}^-]$						&
	$B_j$ & $H [L_{j,k-1}^+, L_{j,k}^+]$					\\\hline
	$r_l$	  & $Q^l(I_2 - Q^N)^{-1}$ 							&
	$\mathcal{M}_{j,m}$	&	$(A_j+B_j)^{-1}(G L_{j,m}^- + H L_{j,m}^+)$	\\\hline
	\end{tabular}
\end{table}

\subsection{Projections}
\label{subsec:proj}
 
We summarize the definition and properties of projections in this subsection.
We denote the standard $L^2$ projection of $u$ onto $V_h^k$ by $\plt.$ Clearly,  
$
\plt u|_{I_j} = \sum_{m=0}^k u_{j,m}L_{j,m}.
$
The following projection $\pst$ was introduced  in \cite{2018arXiv180105875C}.
\begin{definition}
\label{def:pst}
For DG scheme with flux choice \eqref{eqn:flux}, we define the associated projection operator $\pst$ for any periodic function $u \in W^{1,\infty}(I) $  
to be the unique polynomial $\pst u \in V_h^k$ (when $k \geq 1$) satisfying  
\begin{subequations}
\label{eqn:pst}
\begin{align} \label{eqn:psts1}
\intj \pst u  \, \vh dx &=\intj u  \,\vh dx \quad & &\forall \,\vh \in P_c^{k-2}(I_j),\\
\label{eqn:psts2}
  \widehat{\pst u}  =\{\pst u\} -\ao [\pst u] +\bt[(\pst u)_x]   &= u  & &\textrm{at} \quad x_{\jp}, \\
\label{eqn:psts3}
\widetilde{(\pst u)_x}=\{(\pst u)_x\}+\ao [(\pst u)_x] +\bo[\pst u] & =u_x & &\textrm{at} \quad x_{\jp},
\end{align}
\end{subequations}
for all $j \in \zn$. When
$k=1$,  only conditions \eqref{eqn:psts2}-\eqref{eqn:psts3}  are needed.
\end{definition}

\eqref{eqn:psts2}-\eqref{eqn:psts3} is equivalent to
\begin{align}
\label{eqn:psteq}
& G \begin{bmatrix}
\pst u 	\\
(\pst u)_x
\end{bmatrix}
\bigg \rvert_{x_{j+\ot}}^-
+
H \begin{bmatrix}
\pst u 	\\
(\pst u)_x
\end{bmatrix}
\bigg \rvert_{x_{j+\ot}}^+
=
G
\begin{bmatrix}
 u 	\\
u_x
\end{bmatrix}
\bigg \rvert_{x_{j+\ot}}
+
H \begin{bmatrix}
 u 	\\
u_x
\end{bmatrix}
\bigg \rvert_{x_{j+\ot}}.
\end{align}

In \cite{2018arXiv180105875C}, the properties of $\pst$ with general parameter choice $\ao, \bo, \bt$ are shown, which is a key step to establish optimal convergence of the UWDG for many cases. 
In this paper, we only consider $k \geq 2,$ which is required for superconvergence properties to hold.
 
For completeness of the paper, we will briefly summarize the properties of $\pst$ under the our assumptions as follows.  
Based on the results in \cite{2018arXiv180105875C}, the existence and uniqueness of $\pst$ is guaranteed if any of the  following assumptions is satisfied. 
\begin{itemize}
\item A1. (Local projection) scale invariant flux, $\ao^2 + \bo \bt = \frac{1}{4}$ and $\Gamma_j \neq 0$.
\item A2. (Global projection) scale invariant flux,  uniform mesh ($h_j=h, \forall j$), $\ao^2 + \bo \bt \neq \frac{1}{4}$ and $\abs{ \frac{\g}{\la}} > 1$.  
\item A3. (Global projection) scale invariant flux, uniform mesh ($h_j=h, \forall j$), $\ao^2 + \bo \bt \neq \frac{1}{4}$. Besides, either $\abs{ \frac{\g}{\la}} = 1$, $\left ( (-1)^{k+1}\frac{\g}{\la}\right )^N \neq 1$ with $N$ being an odd number or $\abs{ \frac{\g}{\la}} < 1$, $\left ( (-1)^{k+1}\frac{\g}{\la} + \sqrt{\left (\frac{\g}{\la} \right )^2 - 1} \right )^N \neq 1.$
\end{itemize}

For example, alternating fluxes satisfy A1, central flux satisfies A2, IPDG, DDG and other more general  fluxes may satisfy any of the assumptions A1/A2/A3 depending on the parameters.
Here, the word ``local" refers to the fact that $\pst$ can be locally determined on each cell $I_j$. Otherwise $\pst$ is a global projection and $\pst u$ is the solution of a $2N \times 2N$ block-circulant system, where the coefficient matrix   is
\beq
\label{eqn:mdef}
M=circ(A,B,0_2,\cdots,0_2),
\eeq
denoting a block-circulant matrix with first two rows as $(A,B,0_2,\cdots,0_2)$, $0_2$ is the $2\times 2$ zero matrix, and $A, B$ are defined in  Table \ref{tab:notation}.

  Assumptions A2 and A3 ensure the global matrix $M$ is invertible.  It's known that
\beq
\label{eqn:minv}
M^{-1} = circ(r_0, \cdots, r_{N-1}) \otimes A^{-1},
\eeq
 where $\otimes$ means Kronecker product for block matrices, 
 \beq
 \label{eqn:rj}
 r_j=Q^j (I_2-Q^N)^{-1}, \quad Q = -A^{-1}B.
 \eeq
With assumption A2, the eigenvalues of $Q$ are real and distinct. With assumption A3, the eigenvalues are either complex or repeated. 

To shorten the notation, from here on  we use two notations $C_{m}$ and $C_{m,n}$ to denote mesh independent constants.
$C_m$ may depend on  $| u |_{W^{k+1+m,\infty}(I)}$ for assumptions A1/A2,  and  $\| u \|_{W^{k+3+m,\infty}(I)}$ for assumption A3 when $\abs{ \frac{\g}{\la}} < 1$, $\| u \|_{W^{k+4+m,\infty}(I)}$ for assumption A3 when $\abs{ \frac{\g}{\la}} =1$. 
$C_{m,n}$ may depend on  $| u |_{W^{k+1+m+2n, \infty}(I)}$ for assumptions A1/A2, on $\| u \|_ {W^{k+3+m+3n, \infty}(I)}$ for assumption A3 when $\frac{|\g|}{|\la|} < 1$ and on $ \| u \|_ {W^{k+4+m+4n, \infty}(I)}$ for assumption A3 when $\frac{|\g|}{|\la|} = 1$.

We have the following  lemma for some algebraic formulas that will be used several times in later sections.
\begin{Lem}
\label{lem:mjm}
Suppose any of the assumptions A1/A2/A3 holds. Define
\beq
\label{eqn:mjmdef}
\mathcal{M}_{j,m} = (A_j+B_j)^{-1}(G L_{j,m}^- + H L_{j,m}^+), \quad \forall m \in \mathbb{Z}^+, \forall j \in \zn,
\eeq
where $G, H, A_j, B_j, L_{j,m}^-, L_{j,m}^+, \g_j, \la_j$ are defined   in Table \ref{tab:notation}. Then $\forall j \in \zn$,
\beq
\label{eqn:mjm}
\left \| (A_j+B_j)^{-1} G  \bmat 1 & 0 \\ 0 & \frac{1}{h_j} \emat \right \|_\infty \leq C, \ \left \| (A_j+B_j)^{-1} H  \bmat 1 & 0 \\ 0 & \frac{1}{h_j} \emat \right \|_\infty \leq C, \  \left \| \mathcal{M}_{j, m} \right \|_\infty \leq C.
\eeq
\end{Lem}

\begin{proof}
If the mesh is uniform, the three matrices in \eqref{eqn:mjm} are independent of mesh size $h$ and the inequalities in \eqref{eqn:mjm} hold.
For nonuniform mesh, the proof is given in Appendix \ref{sec:mjmproof}.
\end{proof}


Since $\pst u \in V_h^k$, it can be expressed  in Legendre basis as
\beq
\label{eqn:pstcoef}
\pst u|_{I_j} = \sum_{m=0}^{k} \acute u_{j,m} L_{j,m}.
\eeq
By \eqref{eqn:psts1}, $u - \pst u \perp V_h^{k-2}$, thus $\acute u_{j,m} = u_{j,m}, \forall m \leq k-2$. 
 
The following Lemma is a summary and extension of the results in Lemmas 3.2, 3.4, 3.8, 3.9 in \cite{2018arXiv180105875C}.
\begin{Lem}
\label{lem:pstest}  
Suppose any of the assumptions A1/A2/A3 holds, for $u$ satisfying the condition in Definition \ref{def:pst}. If assumption A1 is satisfied, then \eqref{eqn:psteq} is equivalent to $\forall j \in \zn$,
\beq
\label{eqn:psteq2}
G \begin{bmatrix}
\pst u 	\\
(\pst u)_x
\end{bmatrix}
\bigg \rvert_{x_{j+\ot}}^-
+
H \begin{bmatrix}
\pst u 	\\
(\pst u)_x
\end{bmatrix}
\bigg \rvert_{x_{j-\ot}}^+
=
G
\begin{bmatrix}
 u 	\\
u_x
\end{bmatrix}
\bigg \rvert_{x_{j+\ot}}
+
H \begin{bmatrix}
 u 	\\
u_x
\end{bmatrix}
\bigg \rvert_{x_{j-\ot}},
\eeq
thus making $\pst$ a local projection and
\beq
\label{eqn:pstcoefa1}
\bmat
\acute u_{j,k-1}	\\
\acute u_{j,k}
\emat
=
\begin{bmatrix}
u_{j,k-1}	\\
u_{j,k}
\end{bmatrix}
+
\sum_{m=k+1}^\infty u_{j,m} \mathcal{M}_{j, m}.
\eeq
If any of the assumptions A2/A3 is satisfied, then
\beq
\label{eqn:pstcoefa23}
\bmat
\acute u_{j,k-1}	\\
\acute u_{j,k}
\emat
=
\bmat
u_{j,k-1}	\\
u_{j,k}
\emat
+ \sum_{m=k+1}^\infty \big (u_{j,m} V_{1,m} +  \suml u_{j+l,m} r_l V_{2,m} \big ),
\eeq
where $V_{1,m} = [ L_{k-1}^+, L_k^+ ] ^{-1} L_m^+$, $V_{2,m} = [ L_{k-1}^- , L_k^- ]^{-1} L_m^- - [ L_{k-1}^+, L_k^+ ]^{-1} L_m^+$.

We have the following estimates
\beq
\label{eqn:pstest}
\abs{\acute u_{j,m} - u_{j,m} } \leq C_0 h^{k+1}, m = k-1, k, \quad \| u - \pst u \|_{L^{\nu}(\mathcal I_N)} \leq C_0 h^{k+1}, \quad \nu = 2, \infty.
\eeq
In addition, if $h_j = h_{j+1}$,
\beq
\label{eqn:pstcoefdiff}
\abs{\acute u_{j,m} - u_{j,m} - (\acute u_{j+1,m} - u_{j+1,m}) } \leq C_1 h^{k+2}, \quad m=k-1,k.
\eeq
\end{Lem}

\begin{proof}
Proof is given in Appendix \ref{apdx:pstproof}.
\end{proof}

With the optimal estimates of $\pst u$, we proved the optimal $L^2$ error estimate of the DG scheme in Theorem 3.10 in \cite{2018arXiv180105875C}, which is restated below.
\begin{Thm}\cite{2018arXiv180105875C}
\label{thm:l2converge}
Suppose any of the assumptions A1/A2/A3 holds, let the exact solution $u$ of \eqref{eqn:ls} be sufficiently smooth, satisfying periodic boundary condition and $u_h$ be the DG solution in \eqref{eqn:scheme}, then
\beq
\label{eqn:l2converge}
\| \pst u - u_h \| \leq C_2 h^{k+1}, \quad \| u - u_h \| \leq C_2 h^{k+1}.
\eeq
\end{Thm}

 \bigskip
Next, we introduce a local projection $\pdag$ as a variant of $\pst$ 
and study its approximation properties. Similar ideas have been employed
 in \cite{cao2017superconvergence} for proving the superconvergence at 
 the so-called generalized Radau points when using upwind-biased flux 
 for hyperbolic equations. This projection will help us reveal the superconvergence results at special points.

\begin{definition}
\label{def:pdag}
For DG scheme with flux choice \eqref{eqn:flux}, we define a local projection operator $\pdag$ associated with flux choice \eqref{eqn:flux} for any periodic function $u \in W^{1,\infty}(I) $ to be the unique polynomial $\pdag u \in V_h^k$ (when $k \geq 1$) satisfying 
\begin{subequations}
\label{eqn:pdag}
\begin{align} \label{eqn:pdag1}
\intj \pdag u  \, \vh dx &=\intj u  \,\vh dx,  \quad \forall \vh \in P_c^{k-2}(I_j),\\
\label{eqn:pdagflux}
G \begin{bmatrix}
\pdag u 	\\
(\pdag u)_x
\end{bmatrix}
\bigg \rvert_{x_{j+\ot}}^-
+
H  \begin{bmatrix}
\pdag u 	\\
(\pdag u)_x
\end{bmatrix}
\bigg \rvert_{x_{j-\ot}}^+
& =
G
\begin{bmatrix}
 u 	\\
u_x
\end{bmatrix}
\bigg \rvert_{x_{j+\ot}}
+
H \begin{bmatrix}
 u 	\\
u_x
\end{bmatrix}
\bigg \rvert_{x_{j-\ot}}
\end{align}
\end{subequations}
for all $j \in \zn$. When $k=1$,  only condition \eqref{eqn:pdagflux} is needed.
\end{definition}

Projection $\pdag$ is   always a local projection. Denote $\pdag u|_{I_j} = \sum_{m=0}^k \grave u_{j, m} L_{j, m}$, by \eqref{eqn:pdag1}, $\grave u_{j, m} = u_{j, m}, m \le k-2$. The similarities in definition imply that $\pst$ and $\pdag$ are very close to each other, as shown in the following lemma. 


\begin{Lem}
\label{lem:pdagprop}  
For periodic function $u \in W^{1,\infty}(I)$, if assumption A1 is satisfied, $\pst u = \pdag u.$ If any of the assumptions A2/A3 is satisfied, $\pdag$ exists and is uniquely defined if 
$(-1)^{k+1}\frac{\g_j}{\la_j} \neq 1$ for all $j \in \zn$. Then,
\beq
\label{eqn:pdagest}
\| u - \pdag u \|_{L^\nu (I_j)}  \leq Ch^{k+1} |u |_{W^{k+1,\nu}(I_j)}, \quad \nu =2, \infty.
\eeq

If any of the assumptions A2/A3 is satisfied, we have 
\beq
\label{eqn:projdiff}
\| \pst u - \pdag u \|_{L^\nu (\mathcal{I_N})}   \leq  C_1 h^{k+2},	\quad \nu =2, \infty.
\eeq

\end{Lem}

\begin{proof}
	When assumption A1 is satisfied, due to \eqref{eqn:psteq2}, $\pst = \pdag$. The rest of the proof is given in Appendix \ref{apdx:pdagproof}.
\end{proof}

To analyze the superconvergence property at special points, we need 
 to investigate the expansion of the projection error of $\pdag$ on every cell $I_j$
\beq
\label{eqn:updagdiff}
\begin{split}
(u -\pdag u)|_{I_j}  & = [L_{j,k-1}, L_{j,k} ] \bmat u_{j,k-1} - \grave u_{j,k-1} \\ u_{j,k} - \grave u_{j,k} \emat  + \sum_{m = k+1}^\infty u_{j,m} L_{j,m} = \sum_{m=k+1}^\infty u_{j,m} R_{j,m},	\quad \text{if } A1/A2/A3,
\end{split}
\eeq
where $u_{j,m}$ is defined in \eqref{eqn:ltcoefcompute} and 
\beq
\label{eqn:rjm}
R_{j,m} = L_{j,m} - [L_{j,k-1}, L_{j,k}]\mathcal{M}_{m}.
\eeq 
We write out the explicit expression of the leading term in expansions
\beq
\label{eqn:rjk1}
R_{j,k+1}=L_{j,k+1} + b L_{j,k} + c L_{j,k-1},
\eeq
 where
\begin{align*}
b & = - \frac{2\ao \frac{2k+1}{h_j}}{ \g_j + (-1)^k \la_j}, \\
c & = -\frac{\bo - \frac{2(k+1)^2}{h_j}(\ao^2 + \bo\bt + \frac{1}{4}) - (-1)^{k+1} \frac{2(k+1)}{h_j}(\ao^2 + \bo\bt - \frac{1}{4})  + \frac{\bt}{{h_j}^2} k(k+2)(k+1)^2}{ \g_j + (-1)^k \la_j}
\end{align*}
to determine the location of superconvergent points.

%
%
 
For $s = 0, 1, 2$, denote $D^s_j$ as the roots of $\frac{d^s}{dx^s} R_{j,k+1},$ $D^s = \bigcup_{j=1}^N D_{j}^s$, then it follows from \eqref{eqn:updagdiff} and \eqref{eqn:ltcoef} that, for $x \in D^s_j$, 
\beq
\label{eqn:dsj}
\begin{split}
\partial^s(u - \pdag u)(x) &= \sum_{m=k+2}^\infty u_{j,m} \frac{d^s}{dx^s} R_{j,m} \leq Ch^{k+\frac{3}{2} - s} | u|_{W^{k+2,s}(I_j)}, \quad \text{if } A1/A2/A3,
\end{split}
\eeq
indicating superconvergence at those points. For details, please see results in Theorem \ref{thm:uhprojdiff}.

  Since the expression of $b, c$ depends on $h_j$,  on nonuniform mesh, $D^s_j, s = 0, 1, 2$ have nodes with the different relative locations on each cell. For simplicity, below we discuss the locations of $D^0_j, D^1_j, D^2_j$ for special flux choices on uniform mesh. 
\begin{itemize}
\item Alternating fluxes: $b = \pm \frac{2k+1}{k}, c= -\frac{(k+1)^2}{k}$.
\item Central flux: if $k$ is even, then $ b= 0, c = -\frac{(k+1)(k+2)}{k(k-1)}$; if $k$ is odd, then $b = 0, c = -1$.
\item IPDG fluxes: if $k$ is even, then $b = 0, c = -\frac{(k+1)(k+2) - 2 \tilde \bo}{k(k-1) - 2 \tilde \bo}$; if $k$ is odd, then $b= 0, c = -1$.
\end{itemize}

For central and IPDG fluxes, if $k$ is odd, $R_{j,k+1} = L_{j,k+1} - L_{j,k-1}, \frac{d}{dx}R_{j,k+1} = \frac{4k+2}{h_j}L_{j,k}, \frac{d^2}{dx^2}R_{j,k+1} = \frac{8k+4}{h_j^2}L'_{j,k}, $ implying that $D^0_j, D^1_j, D^2_j$ are Lobatto points of order $k+1$, Gauss points of order $k$ and Lobatto points of order $k+1$ excluding end  points, respectively, on interval $I_j.$  Therefore, $card(D^0_j) = k+1, card(D^1_j) = k, card(D^2_j) = k-1$.

Cao \emph{et al.} proved there exists $k+1$ superconvergence points (Radau points) when using upwind flux for linear hyperbolic problem in \cite{cao2014superconvergence}, $k+1$ superconvergence points (Lobatto points) using special flux parameter in DDG method in \cite{cao2017superconvergenceddg} and $k+1$ or $k$ superconvergence points, depending on parameters, for using upwind-biased flux for linear hyperbolic problem in \cite{cao2017superconvergence}.
Analyzing the number and location of superconvergent points for our scheme is more challenging. We shall only provide lower bound estimates for the number of superconvergence points.
For general parameters choices, when $k \geq 2$, $R_{j,k+1} \perp P_c^{k-2}(I_j)$, by Theorem  3.3 and Corollary 3.4 in \cite{shen2011spectral}, we can easily show $R_{j,k+1}$ has at least $k-1$ simple zeros, i.e., $card(D^0_j) \geq k-1$. By the same approach, we can show when $k \geq 3$, $card(D^1_j) \geq k-2,$  and when $k \geq 4$, $card(D^2_j) \geq k-3$. For small $k$ values, $D^1_j, D^2_j$ can possibly be empty sets.

\section{Superconvergence Properties}
\label{sec:suplinear}
In this section, we study  superconvergence of the numerical solution. First, we  investigate the superconvergence of UWDG fluxes, cell averages, towards a particular projection and at some special points. This 
analysis is done by decomposing the error into
\begin{equation}
\label{eqn:decompose}
e = u - u_h = \epsilon_h + \zeta_h, \ \epsilon_h = u - u_I, \  \zeta_h = u_I - u_h
\end{equation} for some $u_I \in V_h^k$. For error analysis of DG schemes, $u_I$ is usually taken as some projection of $u.$ While for our purpose of superconvergence analysis, $u_I$ needs to be carefully designed as illustrated in Section \ref{sec:correction}.   First, without specifying $u_I,$ we prove some intermediate superconvergence results in Section \ref{sec:int}. Then, the choice of $u_I$ is made in Section \ref{sec:correction} and the main results are obtained.
The other superconvergence property is about the negative norm of the UWDG solution, which enables a post-processing technique to obtain highly accurate solutions. This is considered in Section \ref{sec:postp}.

\subsection{Some intermediate superconvergence results}  
\label{sec:int}

This subsection will collect superconvergence results of $\| (\zeta_h)_{xx} \|_{L^2(I)}, (\frac{1}{N} \sumj |[{\zeta_h}]|^2_{\jp})^{\ot},$ $(\frac{1}{N} \sumj |[(\zeta_h)_x]|_{\jp}^2)^{\ot}$ without specifying $u_I.$ The main idea is to choose special test functions in error equation, similar to the techniques used in \cite{cheng2010superconvergence} for hyperbolic problems. This is an essential step to elevate the superconvergence order in Theorem \ref{lem:zetaest} when $k$ is even.   


\begin{Lem} 
\label{lem:zetaxx}
For $ k \geq 2$, let $u$ be the exact solution to \eqref{eqn:ls} and $u_h$ be the DG solution in \eqref{eqn:scheme}. $\epsilon_h, \zeta_h$ are defined in \eqref{eqn:decompose}. We define $s_h$ to be a unique function in $V_h^k,$ such that $\int_I s_h v_h dx = a(\epsilon_h, v_h), \, \forall v_h \in V_h^k$.  Then, when any of the assumptions A1/A2 is satisfied, we have
\begin{align}
\| (\zeta_h)_{xx} \|  & \leq C\|s_h + (\zeta_h)_t\| ,			\label{eqn:zetaxx} \\
(\frac{1}{N} \sumj |[\zeta_h]|_{\jp}^2)^{\ot}  	&\leq Ch^2 \|s_h + (\zeta_h)_t\|,	\label{eqn:zetajump} \\
(\frac{1}{N} \sumj |[(\zeta_h)_x]|_{\jp}^2)^{\ot}  	& \leq Ch \|s_h + (\zeta_h)_t\|.	\label{eqn:zetaxjump}
\end{align}
\end{Lem}

\begin{proof}
The proof is given in Appendix \ref{apdx:zetaxxproof}.
\end{proof}


\subsection{ Correction functions and the main results}
\label{sec:correction}
In this section, we shall present the main superconvergence results. The proof depends on Lemma \ref{lem:zetaxx} and the correction function technique introduced by Cao \emph{et al.}  in \cite{cao2014superconvergence, cao2017superconvergenceddg}, which is essential for superconvergence. We let $u_I=\pst u$ when $k=2,$ and $u_I = \pst u - w,$ when $k \ge 3,$ where  $w \in V_h^k$ is a specially designed correction function defined below.

Similar to \cite{cao2017superconvergenceddg}, we start the construction by defining $w_q, 1 \leq q \leq \floor{\frac{k-1}{2}}$. For $k \geq 3$, we denote $w_0 = u - \pst u$ and define a series of functions $w_q \in V_h^k,$ as follows
\begin{subequations}
\label{eqn:wqdef}
\begin{align}
\label{eqn:wq1}
\intj w_{q} (v_h)_{xx} dx & = -i \intj (w_{q-1})_t v_h dx,		& \forall v_h \in P_c^{k}(I_j) \setminus P_c^1(I_j),	\\
\label{eqn:wq2}
\widehat{w_q} & = 0,							& \text{at } x_{j+\ot},	\\
\label{eqn:wq3}
\widetilde{(w_q)_x} & = 0,						& \text{at } x_{j+\ot},	
\end{align}
\end{subequations}
for all $j \in \mathbb{Z}_N$. \eqref{eqn:wq2} and \eqref{eqn:wq3} is equivalent to
\begin{align}
\label{eqn:wqvec}
& G \begin{bmatrix}
w_q 	\\
(w_q)_x
\end{bmatrix}
\bigg \rvert_{x_{j+\ot}}^-
+
H \begin{bmatrix}
w_q 	\\
(w_q)_x
\end{bmatrix}
\bigg \rvert_{x_{j+\ot}}^+
=
0.
\end{align}
$w_q$ exists and is unique when any of the assumptions A1/A2/A3 is satisfied for the same reason as the existence and uniqueness of $\pst$. 

With the construction of $w_q$, we define 
\beq
\label{eqn:wdef}
w(x,t) = \sum_{q=1}^{\floor{\frac{k-1}{2}}} w_q(x,t),
\eeq
then
\beq
\begin{aligned}
\label{eqn:aui}
a_j(\epsilon_h, v_h) 	& = a_j(u-\pst u, v_h) + \sum_{q=1}^{\floor{\frac{k-1}{2}}} a_j(w_q, v_h)	\\
			& = \intj (w_0)_t v_h dx + \sum_{q=1}^{\floor{\frac{k-1}{2}}}  \left ( \intj (w_q)_t v_h dx - i \intj w_q (v_h)_{xx} dx \right )\\
			&=  \intj (w_0)_t v_h dx + \sum_{q=1}^{\floor{\frac{k-1}{2}}}  \intj (w_q-w_{q-1})_t v_h dx \\
			& = \intj (w_{\floor{\frac{k-1}{2}}})_t v_h dx,   \quad \forall v_h \in V_h^k(I).
\end{aligned}
\eeq

The approximation property of $w_q$ and $a_j(\epsilon_h, v_h)$ are presented in the following Lemma.

\begin{Lem}
\label{lem:west}
For $k \geq 3$, suppose $u$ satisfies the condition in Theorem \ref{thm:l2converge}. For $w_q, 1 \leq q \leq \floor{\frac{k-1}{2}}$, $q+r \leq  \floor{\frac{k-1}{2}} +1,$ we have
\beq
\label{eqn:wqexp}
\begin{aligned}
\partial_t^r w_q|_{I_j} = \sum_{m = k - 1 - 2q}^k \partial_t^r c_{j,m}^q L_{j,m}, 
\quad \partial_t^r c_{j, k-1-2q}^q = C h_j^{2q} \partial_t^{q+r} (u_{j,k-1} - \acute u_{j,k-1}), \\
\quad \abs{\partial_t^r c_{j, m}^q} \leq C_{2r,q} h^{k+1+2q},
\end{aligned}
\eeq
and then
\beq
\label{eqn:west2}
\| \partial_t^r w_q \|\leq C_{2r,q}h^{k+1+2q}. 
\eeq
For any $v_h \in V_h^k$,
\beq
\label{eqn:wqerreq}
\abs{a(\epsilon_h, v_h)} \leq C_{2,\floor{\frac{k-1}{2}}} h^{k+1+2\floor{\frac{k-1}{2}}}\|v_h\|.
\eeq
\end{Lem}
\begin{proof}
	The proof is given in Appendix \ref{apdx:westproof}.
\end{proof}

\begin{Lem}
\label{lem:zetaxx2}
For $k \geq 2$, suppose $u$ satisfies the condition in Theorem \ref{thm:l2converge}. If the parameters satisfy any of the assumptions A1/A2 and $u_h|_{t=0} = u_I|_{t=0}$, we have
\begin{align}
\label{eqn:uiuh1}
\| (\zeta_h)_{xx} \| & \leq C_{4+2\floor{\frac{k-1}{2}}} h^{k+1+2\floor{\frac{k-1}{2}}}.	\\
\label{eqn:uiuh2}
(\frac{1}{N} \sumj |[(\zeta_h)]|_{\jp}^2)^{\ot}  & \leq C_{4+2\floor{\frac{k-1}{2}}} h^{k+3+2\floor{\frac{k-1}{2}}}.	\\
\label{eqn:uiuh3}
(\frac{1}{N} \sumj |[(\zeta_h)_x]|_{\jp}^2)^{\ot}  	& \leq C_{4+2\floor{\frac{k-1}{2}}} h^{k+2+2\floor{\frac{k-1}{2}}}.
\end{align}
\end{Lem}

\begin{proof}
When $k = 2$,   $w=0.$ $a(\epsilon_h, v_h)  = \int_I (\epsilon_h)_t v_h dx$  from the definition of $\pst$. That is, $s_h = (\epsilon_h)_t $ in the condition of Lemma \ref{lem:zetaxx}.  To bound $\|(\zeta_h)_t\|,$ we take the time derivative of the error equation and obtain
\[
	a(e_t, v_h) =  a((\epsilon_h)_t, v_h) + a((\zeta_h)_t, v_h) = 0.
\]

Let $v_h = \overline{(\zeta_h)_t}$, by \eqref{eqn:l2stab} and the property of $\pst u$, we obtain
\[
	\frac{d}{dt} \| (\zeta_h)_t\|^2  = -a((\epsilon_h)_t, \overline{(\zeta_h)_t}) - \overline{a((\epsilon_h)_t, \overline{(\zeta_h)_t})} \leq 2 \| (\epsilon_h)_{tt} \| \| (\zeta_h)_t \|, 
\]
which implies $ \frac{d}{dt} \| (\zeta_h)_t\| \le  \| (\epsilon_h)_{tt}\|. $
To estimate $\|(\zeta_h)_t|_{t=0}\|$, we let $t=0$ in the error equation. Since $\zeta_h |_{t=0} = (u_h - u_I)|_{t=0}  = 0$, we have

\[
	a(\epsilon_h,v_h) + \int_I (\zeta_h)_t|_{t=0} v_h dx = 0.
\]

Let $v_h = \overline{(\zeta_h)_t}|_{t=0}$, then
$$
\|(\zeta_h)_t|_{t=0}\|^2 \leq \| (\epsilon_h)_{t} \| \|(\zeta_h)_t|_{t=0}\|.
$$
Therefore, 
$$
\|(\zeta_h)_t\| \leq \| (\epsilon_h)_{t} \|+t \| (\epsilon_h)_{tt}\|.
$$
By Lemma \ref{lem:zetaxx}, estimates in \eqref{eqn:pstest} and the inequality above, we can get \eqref{eqn:uiuh1}-\eqref{eqn:uiuh2}.

For $k \geq 3$, by \eqref{eqn:aui}, we have $a(\epsilon_h, v_h) = \int_I (w_{\floor{\frac{k-1}{2}}})_t v_h dx,$ that is, $s_h = (w_{\floor{\frac{k-1}{2}}})_t$ in the condition of Lemma \ref{lem:zetaxx}.    Then, following the same lines of proof as above, by replacing $\epsilon_h$ with $w_{\floor{\frac{k-1}{2}}}$ and using Lemma \ref{lem:west}, we are done.
%
%
%
%
\end{proof}


Now we are ready to state the following estimates of $\| \zeta_h \|$.  
\begin{Thm}
\label{lem:zetaest}
For $k \geq 2$, suppose $u$ satisfies the condition in Theorem \ref{thm:l2converge}. Assume $u_h|_{t=0} = u_I|_{t=0}$, then $\forall t \in (0, T_e]$,
\beq
\label{eqn:uiblue}
\| \zeta_h \| \leq \begin{cases} C_{2,{\frac{k-1}{2}}} h ^{2k} & \text{if k is odd and A1/A2/A3},		\\
( C_{k+2} h^{4k} + \sum_{I_j \subset {I^{NU}}}   C_k h^{4k-1} )^{\ot}  & \text{if k is even and A1/A2}, \\
C_{2,{\frac{k-2}{2}}} h ^{2k-1}, & \text{if k is even and A3},
 \end{cases}
 \eeq
where $I^{NU}$ is the collection of cells in which the length of $I_j$ is different with at least one of its neighbors.
\end{Thm}

\begin{proof}
	\ifshowproof 
	From error equation,  
	$ 
	a(e, \overline{\zeta_h}) =  a(\epsilon_h, \overline{\zeta_h}) + a(\zeta_h, \overline{\zeta_h}) = 0,$  which gives us
	\beq
	\label{eqn:dtzeta}
	\frac{d}{dt} \| \zeta_h\|^2  = -a(\epsilon_h, \overline{\zeta_h}) - \overline{a(\epsilon_h, \overline{\zeta_h})} \leq
	\begin{cases} 
	2 \| (\epsilon_h)_{t} \| \|\zeta_h\|, & k=2,		\\
	2 \| (w_{\floor{\frac{k-1}{2}}})_t  \| \|\zeta_h\|, & k \ge 3.			  \end{cases}
	\eeq
	By \eqref{eqn:pstest}, \eqref{eqn:west2} and Gronwall's inequality, we have
	\[
	 \| \zeta_h\|  \leq C_{2,\floor{\frac{k-1}{2}}} t h^{k+1+2\floor{\frac{k-1}{2}}}, \quad \forall t \in (0, T_e].
	\]
	Therefore, when $k$ is odd ,or $k$ is even and parameters satisfy A3, the proof is complete.
	
	When $k$ is even and parameters satisfy any of the assumptions A1/A2, we make use of Lemma \ref{lem:zetaxx2} to show the improved estimates. We let $l=\floor{\frac{k-1}{2}} = \frac{k-2}{2}$, then
	\[
	\begin{split}
	a(\epsilon_h, \overline{\zeta_h}) & = \int_I (w_l)_t \overline{\zeta_h} dx = \sumj \sum_{m=1}^k \partial_t c_{j,m}^l \intj L_{j,m} \overline{\zeta_h} dx							\\
	& = \sumj \partial_t c_{j,1}^l \intj L_{j,1} \overline{\zeta_h} dx + \sumj \sum_{m=2}^k \partial_t c_{j,m}^l \intj L_{j,m} \overline{\zeta_h} dx  \doteq  \mathcal{A}_1+\mathcal{A}_2,
	\end{split}
	\]
	where we denote the first term in the summation by $\mathcal{A}_1,$ and the other summation term as $\mathcal{A}_2.$  
	Note that $D^{-1} L_{j, m} \perp P^0, m \geq 1$ in the inner product sense, thus $D^{-2} L_{j, m} (\pm 1) =0, m \geq 2$.
	By integration by parts, we get
	\begin{align*}
	\mathcal{A}_2 & = \sumj \sum_{m=2}^k \frac{h_j^2}{4} \partial_t c_{j,m}^l \intj D^{-2} L_{j,m} \overline{(\zeta_h)}_{xx} dx \\
	& \leq C h^{-1} \sumj | \frac{h_j^2}{4} \partial_t c_{j,m}^l|^2  + h \sumj \sum_{m=2}^k (\intj D^{-2} L_{j,m} \overline{(\zeta_h)}_{xx} dx )^2 \leq  C_k h^{4k} + C h^2   \| (\zeta_h)_{xx} \|^2  \le C_{k+2} h^{4k} 
	\end{align*}
	where we have used \eqref{eqn:wqexp} in the first inequality, and \eqref{eqn:uiuh1}  in the third inequality.
	
	To estimate $\mathcal{A}_1$, we take the first and second antiderivative of $L_{j,1} = \xi$ as $\frac{h_j}{2}(\frac{\xi^2}{2}- \frac{1}{6})$, $(\frac{h_j}{2} )^2\frac{\xi^3 - \xi }{6}$ and apply integration by parts twice, 
	\begin{align*}
	\mathcal{A}_1  & =  \sumj \frac{h_j}{2} \partial_t c_{j,1}^l  \left ((\frac{\xi^2}{2} - \frac{1}{6})\bar  \zeta_h \big |_{x_{\jm}}^{x_{\jp}} - \intj (\frac{\xi^2}{2} - \frac{1}{6}) (\bar  \zeta_h)_x dx \right)	\\
	& = \sumj \frac{h_j}{2} \partial_t c_{j,1}^l  \left ( \frac{1}{3} ( \bar \zeta_h|^-_{\jp} - \bar  \zeta_h|^+_{\jm}) + \frac{h_j}{2} \intj \frac{\xi^3-\xi}{6} (\bar  \zeta_h)_{xx} dx \right)	\\
	& = \sumj \frac{h_j}{2} \partial_t c_{j,1}^l \left ( -\frac{1}{3} [\bar \zeta_h]_{\jp}  + \frac{h_j}{2} \intj \frac{\xi^3-\xi}{6} (\bar  \zeta_h)_{xx} dx \right)	\\
	& + \sumj  \left ( \Big ( \frac{h_j}{2} \Big ) \partial_t  c_{j,1}^l - \Big ( \frac{h_{j+1}}{2} \Big ) \partial_t  c_{j+1,1}^l  \right ) \frac{1}{3} \bar \zeta_h|_{\jp}^+,
	\end{align*}
	where we have used the periodicity in the last equality. Therefore,
	\begin{align*}
	\abs{\mathcal{A}_1 } & \leq \ot h \sumj \left(\Big ( \frac{h_j}{2} \Big ) | \partial_t c_{j, 1}^l| \right )^2 + h^{-1} \sumj \frac{1}{9}|[\bar \zeta_h]|^2_{\jp}	 
	 + C h^2  \sumj  \| (\zeta_h)_{xx} \|_{L^2(I_j)}^2		\\
	& + \frac{1}{18} h^{-1}\sumj \abs{ \left ( \Big ( \frac{h_j}{2} \Big ) \partial_t  c_{j,1}^l - \Big ( \frac{h_{j+1}}{2} \Big ) \partial_t  c_{j+1,1}^l \right )}^2 + h \sumj \|\zeta_h\|_{L^2(\partial I_j)}^2	\\
	& \leq  C_{k+2} h^{4k}+ C \| \zeta_h\|^2 + C h^{-1}\sumj \abs{ \left ( \Big ( \frac{h_j}{2} \Big ) \partial_t  c_{j,1}^l - \Big ( \frac{h_{j+1}}{2} \Big ) \partial_t  c_{j+1,1}^l \right )}^2
	\end{align*}
	where we used \eqref{eqn:wqexp}, inverse inequality, \eqref{eqn:uiuh1} and \eqref{eqn:uiuh2} in the last inequality.
	
	
	We estimate the last term in $\mathcal{A}_1$ by the
	estimation of the difference of $u_{j,m}$ in neighboring cells, similar to that in Proposition 3.1 of \cite{XingKdvDG}. 
	If $h_j \neq h_{j+1}$, then
	\[
	\Big ( \frac{h_j}{2} \Big ) \partial_t  c_{j,1}^l - \Big ( \frac{h_{j+1}}{2} \Big ) \partial_t  c_{j+1,1}^l \leq C_{k} h^{2k}.
	\]
	
	If $h_j = h_{j+1}$, by \eqref{eqn:wqexp} and \eqref{eqn:pstcoefdiff},
	\[
	\Big ( \frac{h_j}{2} \Big ) \partial_t  c_{j,1}^l - \Big ( \frac{h_{j+1}}{2} \Big ) \partial_t  c_{j+1,1}^l = C \Big ( \frac{h_j}{2} \Big )^{2l+1} \partial_t^{l+1} (u_{j,k-1} - \acute u_{j,k-1} - (u_{j+1,k-1} - \acute u_{j+1,k-1})) \leq C_{k+1} h^{2k+1}.
	\]
	
	Therefore, we have
	$$
	| \mathcal{A}_1 | \leq  C_{k+2}  h^{4k} + C \| \zeta_h\|^2 + \sum_{I_j \subset {I^{NU}}}  C_k  h^{4k-1} .
	$$
	
	Combine with the estimates for $\mathcal{A}_2,$  we have
	$$
	\frac{d}{dt} \| \zeta_h \|^2 \leq C_{k+2}  h^{4k} + C \| \zeta_h\|^2 + \sum_{I_j \subset {I^{NU}}}  C_k  h^{4k-1}. 
	$$
	By Gronwall's inequality and the numerical initial condition, we obtain
	$$
	\| \zeta_h \| \leq ( C_{k+2} h^{4k} + \sum_{I_j \subset {I^{NU}}}   C_k h^{4k-1} )^{\ot} .
	$$
	The proof is now complete.
	\else
	{\cb (Proof omitted in not-show-proof mode)}
	\fi
	\end{proof}



With Theorem \ref{lem:zetaest}, we are able to prove the following superconvergence results.

\begin{Thm}[Superconvergence of numerical fluxes and cell averages]
\label{thm:efec}
Let 
\beq
\label{eqn:efefxec}
E_f = \Big ( \frac{1}{N} \sumj (u - \widehat{u_h})|_{j+\ot}^2 \Big )^{\ot}, 
\quad E_{f_x} = \Big ( \frac{1}{N} \sumj (u_x - \widetilde{(u_h)_x})|_{\jp}^2 \Big )^{\ot}, 
\quad E_c = \Big( \frac{1}{N} \sumj \Big \lvert \frac{1}{h_j} \intj u-u_h dx \Big \rvert ^2 \Big )^{\ot}.
\eeq
be the errors in the two numerical fluxes and the cell averages, respectively.
For $k\geq 2$, suppose $u$ satisfies the condition in Theorem \ref{thm:l2converge}. Assume $u_h|_{t=0} = u_I|_{t=0}$, then $\forall t \in (0, T_e]$
\begin{itemize}
\item if $k$ is odd, parameters satisfy any of the assumptions A1/A2/A3, we have
\beq
\label{eqn:efest1}
E_f  \leq C_{2,{\frac{k-1}{2}}}  h^{2k} , \quad  E_{f_x}  \leq C_{2,{\frac{k-1}{2}}} h^{2k-1}, \quad E_c  \leq C_{2,{\frac{k-1}{2}}}  h^{2k} ,
\eeq
\item  if $k$ is even, parameters satisfy any of the assumptions A1/A2, we have
\begin{align}
\label{eqn:efest2}
\hspace{-0.8in}
&E_f  \leq ( C_{k+2} h^{4k} + \sum_{I_j \subset {I^{NU}}}   C_k h^{4k-1} )^{\ot}, \   E_{f_x}  \leq ( C_{k+2} h^{4k} + \sum_{I_j \subset {I^{NU}}}   C_k h^{4k-1} )^{\ot} h^{-1}, \\ 
& E_c  \leq ( C_{k+2} h^{4k} + \sum_{I_j \subset {I^{NU}}}   C_k h^{4k-1} )^{\ot},
\end{align}
where $I^{NU}$ is the collection of cells in which the length of $I_j$ is different with at least one of its neighbors.
\item if $k$ is even and parameters satisfy assumption A3, we have
\beq
\label{eqn:efest3}
E_f  \leq C_{2,{\frac{k-2}{2}}} h^{2k-1} , \quad  E_{f_x}  \leq C_{2,{\frac{k-2}{2}}} h^{2k-2},  \quad E_c  \leq C_{2,{\frac{k-2}{2}}}  h^{2k-1}.
\eeq
\end{itemize}
\end{Thm}

\begin{proof}
We first prove the estimates for $E_f$. By \eqref{eqn:wq2} and the definition of $\pst,$ $\widehat \epsilon_h(x_{j+\ot}) = \widehat{u- u_I}(x_{j+\ot}) =0$, then 
\[
(u - \widehat{u_h})|_{j+\ot} = (\widehat{\zeta_h} )|_{j+\ot} = \big (\{ \zeta_h\} -\ao [\zeta_h] + \bt [(\zeta_h)_x] \big )|_{j+\ot}.
\]
Therefore, by inverse inequality and the fact $\bt= \tilde \bt h$,
\[
(E_f)^2  \leq \Big (\frac{1}{N} \| \zeta_h \|^2_{L^2 (\partial \mathcal I_N)} \Big )^{\ot} \leq C \| \zeta_h \|,
\]
and the desired estimates for $E_f$ is obtained by \eqref{eqn:uiblue}. The estimates for $E_{f_x}$ can be obtained following same lines.

Next, we prove the estimates for $E_c$. If $k$ is odd, then $\intj w_q dx = 0, 1 \leq q \leq \frac{k-3}{2},$ by \eqref{eqn:wqexp} and orthogonality of Legendre polynomials. Thus,
$$
\intj u- u_h dx = \intj u- \pst u + \sum_{q=1}^{\floor{\frac{k-1}{2}}} w_q + \zeta_h dx = \intj w_{\floor{\frac{k-1}{2}}} dx + \intj \zeta_h dx.
$$

By the Cauchy-Schwartz inequality, we have
\[
\Big \lvert  \frac{1}{h_j} \intj u- u_h dx \Big \rvert^2 \leq \frac{1}{h_j}( \|\zeta_h \|_{L^2(I_j)}^2 + \|w_{\frac{k-1}{2}}\|_{L^2(I_j)}^2 ).
\]

If $k$ is even, then $\intj w_q dx = 0, 1 \leq q \leq \frac{k-2}{2}$, by \eqref{eqn:wqexp} and orthogonality of Legendre polynomials. Thus, by similar step, we have
\[
\intj u- u_h dx = \intj \zeta_h dx, \quad \Big \lvert  \frac{1}{h_j} \intj u- u_h dx \Big \rvert^2 \leq \frac{1}{h_j}\|\zeta_h \|_{L^2(I_j)}^2.
\]

Therefore, 
\[
E_c \le C ( \|\zeta_h \|^2 + \|w_{\frac{k-1}{2}}\|^2 )^{1/2} \text{ if } k \text{ is odd, } E_c \le C ( \|\zeta_h \|^2)^{1/2} \text{ if } k \text{ is even},
\]
and the desired estimate for $E_c$ is obtained by \eqref{eqn:uiblue} and \eqref{eqn:west2}.
\end{proof}

\begin{Thm}[Superconvergence towards projections and at special points]
\label{thm:uhprojdiff}
Suppose $u$ satisfies the condition in Theorem \ref{thm:l2converge}. Assume $u_h|_{t=0} = \pst u_0$, then $\forall t \in (0, T_e]$,
\beq
\label{eqn:uhpstdiff}
\|u_h - \pst u\| \leq \begin{cases}( C_{4} h^{4k} + \sum_{I_j \subset {I^{NU}}}   C_2 h^{4k-1} )^{\ot} & k = 2, \text{ if A1 or A2} \\  C_{2,1} (1+t) h^{k+3} & k \geq 3,
\end{cases}
\eeq
where $I^{NU}$ is the collection of cells in which the length of $I_j$ is different with at least one of its neighbors.

If $D^s, s = 0, 1, 2$ defined in \eqref{eqn:dsj} are not empty sets. Let 
\beq
\label{eqn:eusuperconv}
\begin{aligned}
&E_u = \Big (\frac{1}{|D^0|} \sum_ {x \in D^0} | (u-u_h)(x)|^2 \Big )^{\ot}, \quad E_{u_x} = \Big ( \frac{1}{|D^1|} \sum_{x \in D^1}  | (u-u_h)_x(x)|^2 \Big )^{\ot} , \\
& E_{u_{xx}} = \Big ( \frac{1}{|D^2|} \sum_{x \in D^2}  | (u-u_h)_{xx}(x)|^2 \Big )^{\ot}
\end{aligned}
\eeq
be the average point value error for the numerical solution, the derivative
of solution and the second order derivative of solution at corresponding sets of points.
Then
\begin{itemize}
\item if $k = 2$ and any of the assumptions A1/A2 is satisfied, we have
\beq
\label{eqn:eusuperconv1}
\begin{split}
& E_u \leq  ( C_{4} h^{4k} + \sum_{I_j \subset {I^{NU}}}   C_2 h^{4k-1} )^{\ot}, \quad E_{u_x} \leq h^{-1} ( C_{4} h^{4k} + \sum_{I_j \subset {I^{NU}}}   C_2 h^{4k-1} )^{\ot}, \\
& E_{u_{xx}} \leq h^{-2}( C_{4} h^{4k} + \sum_{I_j \subset {I^{NU}}}   C_2 h^{4k-1} )^{\ot}.
\end{split}
\eeq
\item if $k \geq 3$ and any of the assumptions A1/A2/A3 is satisfied, we have
\beq
\label{eqn:eusuperconv2}
E_u \leq  C_{2,1}  h^{k+2}, \quad E_{u_x} \leq C_{2,1}  h^{k+1}, \quad E_{u_{xx}} \leq C_{2,1}  h^{k}.
\eeq

\end{itemize}
\end{Thm}

\begin{proof}
When $k=2,$ we have $u_h-\pst u=-\zeta_h$. If any of the assumptions A1/A2 is satisfied, by \eqref{eqn:uiblue}, we have
\[
\| u_h-\pst u \| \leq ( C_{4} h^{4k} + \sum_{I_j \subset {I^{NU}}}   C_2 h^{4k-1} )^{\ot}.
\]

When $k \geq 3$, to relax the regularity requirement, we follow the same steps in Lemma \ref{lem:west}, and change the definition of $u_I$ to $u_I =\pst u - w_1$. Then $\epsilon_h = u - u_I, \zeta_h = u_I - u_h$ and we obtain
\[
\abs{a(\epsilon_h, v_h)} \leq C_{2,1} h^{k+3}\|v_h\| , \quad \forall v_h \in V_h^k.
\]

By the estimates above, \eqref{eqn:west2} and the error equation, we obtain 
\[
	\frac{d}{dt} \| \zeta_h\|^2 \leq 2 \| (w_1)_t  \| \|\zeta_h\| \leq C_{2,1} h^{k+3}.
\]
By Gronwall's inequality,
\[
\| \zeta_h \| \leq C_{2,1} t h^{k+3} +  \| (\zeta_h)|_{t=0}\| = C_{2,1} t h^{k+3} +  \| w_1|_{t=0}\|\leq  C_{2,1} (1+t) h^{k+3}, \quad \forall t \in (0, T_e],
\]
where the initial numerical discretization is used in the first equality. Since $u_h -\pst u = - \zeta_h - w_1$, it follows that $\forall t \in (0, T_e]$, 
\[
\| u_h -\pst u\| \leq \| \zeta_h \| + \| w_1\| \leq C_{2,1} (1+t) h^{k+3}.
\]
Then the proof for \eqref{eqn:uhpstdiff} is complete.

If any of the assumptions A1/A2/A3 is satisfied, then
\[
\begin{split}
E_u  & \leq \Big (\frac{1}{|D^0|} \sum_{x \in D^0} | (u- \pdag u )(x)|^2 +  | (\pst u - u_h )(x)|^2 +  | (\pst u - \pdag u )(x)|^2 \Big )^{\ot} 	\\
	& \leq C h^{k+2} |u|_{W^{k+2,2}(I)} + C \| \pst u - u_h \| + C \| \pst u - \pdag u \|,
\end{split}
\]
where \eqref{eqn:dsj}, inverse inequality, and \eqref{eqn:projdiff} are used in the last inequality.
Then the estimates for $E_u$ is proven by Lemma \ref{lem:pdagprop} and \eqref{eqn:uhpstdiff}. The estimates for $E_{u_x}$ and $E_{u_{xx}}$ can be proven following the same lines.

\begin{rem}
If the initial discretization is taken as $u_h|_{t=0} = u_I|_{t=0}$, the theorem above
still holds. However, the regularity requirement will be higher.
\end{rem}

%
%
%
\end{proof}

\subsection{Superconvergence after postprocessing} 
\label{sec:postp}
In this section, we analyze the superconvergence property of the postprocessed DG solutions for linear \sch equation \eqref{eqn:ls} on uniform mesh by using negative Sobolev norm estimates. 
The postprocessor was originally introduced in \cite{bramble1977higher, mock1978computation} for finite difference and finite element methods, and later applied  to DG methods in \cite{cockburn2003enhanced}. 
The postprocessed solution is computed by convoluting the numerical solution $u_h$ with a kernel function $K^{\nu, l}_h(x) = \frac{1}{h^d}K^{\nu,l}(\frac{x}{h})$, where $d$ is the number of spatial dimensions, and $l$ is the index of $H^{-l}$ norm we're trying to estimate later.
The convolution kernel has three main properties. 
First, it has compact support, making post processing computationally advantageous. 
Second, it preserves polynomials of degree up to $\nu -1$ by convolution, thus the convergence rate is not deteriorated. 
Third, the kernel $K^{\nu,l}$ is a linear combination of B-splines, which allows us to express the derivatives of kernel by difference quotients (see section 4.1 in \cite{cockburn2003enhanced}).

We give the formula for the convolution kernel when the DG scheme uses approximation space $V_h^k$:
$$
K^{2(k+1),k+1} (x) = \sum_{\gamma = -k}^k k_{\gamma}^{2(k+1),k+1} \psi ^{(k+1)}(x-\gamma),
$$
where $\psi^{(k+1)}$ are the B-spline bases and the computation procedure of coefficients $k_{\gamma}^{2(k+1),k+1}$ can be found in  \cite{ryan2005extension}. Then we can define the postprocessed DG solution as
\beq
\label{eqn:ust}
u^*  = \int_{-\infty}^\infty K_h^{2(k+1),k+1}({y-x} ) u_h (y) dy.
\eeq
$u^*$ is an ``averaged'' version of $u_h$ such that it is closer as an approximation to the exact solution $u$.
Lastly, we define divided difference as
$$
d_h v(x) = \frac{1}{h} (v(x + \ot h) - v(x-\ot h)).
$$

Now we are ready to state an approximation result showing the smoothness of $u$ and negative Sobolev norm of divided difference lead to a bound on $u - u^*$.
\begin{Thm}[Bramble and Schatz \cite{bramble1977higher}] 
\label{thm:bramble}
Suppose 
$u^*$ is defined in \eqref{eqn:ust} and $K_h^{2(k+1),k+1} = \frac{1}{h}K^{2(k+1),k+1}(\frac{x}{h})  $, where $ K^{2(k+1),k+1}$ is a kernel function as defined above. Let $u$ be the exact solution of linear \sch equation \eqref{eqn:ls} satisfying periodic boundary condition, $u \in H^{2k+2}(I)$. Then for arbitrary time $t \in (0, T_e]$, $h$ sufficiently small, we have
\beq
\label{eqn:pp}
\|u -  u^*\| \leq C h^{2k+2} |u|_{H^{2k+2}(I)} + \sum_{\alpha \leq {k+1}} \|d_h^\alpha ( u- u_h)\|_{H^{-(k+1)}(\mathcal{I}_N)},
\eeq
where C is independent of $u$ and $h$.
\end{Thm}

The right hand side of \eqref{eqn:pp} indicates that if $\|d_h^\alpha ( u- u_h)\|_{H^{-(k+1)}(\mathcal{I}_N)}$ converges at a rate higher than $k+1$, then we have superconvergence property for the postprocessed solution. In what follows, we estimate the negative-norm term following the steps in \cite{cockburn2003enhanced}. First, we introduce a dual problem: 
find a function $v$ such that $v(\cdot, t)$ is periodic function with period equal to the length of $I$, i.e., $b-a$ for all $t\in (0,T_e]$ and
\beq
\begin{aligned}
iv_t - v_{xx} = 0, 	&\quad \text{in} \ \  \mathbb{R} \times (0, T_e), 	\\
v(x,T_e) = \Phi(x), 	&\quad x \in \mathbb{R},
\end{aligned}
\eeq
where $\Phi$ is an arbitrary function in $\mathcal{C}_0^\infty (I)$.
We use the notation $(\phi, \psi) := \int_I \phi \, \psi dx$ in this section. At final time $T_e$,
\begin{align*}
(u(T_e) - u_h(T_e), \Phi) 	& = (u, v)(T_e) - (u_h, v)(T_e) 	\\
				& = (u,v)(0) + \int_0^{T_e}	\{ (u,v_t) + (u_t,v) \} dt   - (u_h, v)(T_e)	\\
				& = (u,v)(0)  - (u_h, v)(0) - \int_0^{T_e} \{((u_h)_t, v) + (u_h, v_t)\} dt	\\
				& = (u-u_h,v)(0) - \int_0^{T_e} \{((u_h)_t, v) + (u_h, v_t)\} dt,
\end{align*}
where the property $u v_t + u_t v=0$ is used to obtain the third equality.

The DG solution $u_h$ satisfies \eqref{eqn:scheme}. Therefore, we have $\forall v_h \in V_h^k$
\begin{align*}
((u_h)_t, v) 	&= ((u_h)_t, v-v_h) + ((u_h)_t, v_h) 	\\
			&= ((u_h)_t, v-v_h) + i A (u_h,v_h)	\\
			&= ((u_h)_t, v-v_h) - i A(u_h,v-v_h) + i A(u_h, v).
\end{align*}

Then we obtain 
$$
(u(T_e) - u_h(T_e), \Phi) = \Theta_M + \Theta_N + \Theta_C,
$$
where
\[
\begin{split}
\Theta_M & =  (u-u_h,v)(0),	\\
\Theta_N & = -\int_0^{T_e} \{ ((u_h)_t, v-v_h) - i A(u_h,v-v_h) \} dt, \quad \forall v_h \in V_h^k,	\\
\Theta_C & = -\int_0^{T_e} \{ (u_h, v_t) + i A(u_h,v) \} dt.
\end{split}
\]

By choosing the initial numerical discretization $u_h(0) = P_h^0 u_0$ and $v_h = P_h^0 v$, we have $\Theta_M  =  (u-u_h,v)(0) = (u-u_h,v-v_h)(0)$ and
$$
\abs{\Theta_M} \leq \|(u-u_h)(0) \|  \cdot \|(v-v_h)(0) \|  \leq Ch^{2k+2} \|u\|_{H^{k+1}(I)} \|v\|_{H^{k+1}(I)}.
$$

Since $v$ is a smooth function, we have
$$
\Theta_C = -\int_0^{T_e} \{ (u_h, v_t) +i (u_h, v_{xx}) \} dt = 0.
$$

Choose $v_h = P_h^0 v$ and from the symmetry of the operator $A(\cdot, \cdot),$ we get
\begin{align*}
\abs{\Theta_N}		& = \abs {\int_0^{T_e} A(u_h, v-v_h) dt} = \abs {\int_0^{T_e} A(v-v_h,u_h) dt}			 = \abs {\int_0^{T_e} \sumj \big (\widehat {v-v_h}[(u_h)_x] - \widetilde {(v-v_h)_x}[u_h] \big ) \big \rvert_{\jp} dt } 		\\
				& = \abs {\int_0^{T_e} \sumj \big (\widehat {v-v_h}[u_x - (u_h)_x] - \widetilde {(v-v_h)_x}[u - u_h] \big ) \big \rvert_{\jp} dt}		\\
				& \leq CT_e \max_{t \in (0, T_e]} \big (\|u - u_h \|_{L^2(\partial \mathcal I_N)} \| \widetilde {(v-v_h)_x} \|_{L^2(\partial \mathcal I_N)} +  \|(u - u_h)_x \|_{L^2(\partial \mathcal I_N)} \| \widehat {v-v_h} \|_{L^2(\partial \mathcal I_N)} \big ).	
\end{align*}


By \eqref{eqn:l2converge},
\begin{align*}
\|u - u_h\|_{L^2(\partial \mathcal I_N)} & = \|u - \pst u\|_{L^2(\partial \mathcal I_N)} + \|\pst u - u_h\|_{L^2(\partial \mathcal I_N)} 	\\
	& \leq C_0 h^{k+\ot} + C h^{-\ot}\|\pst u - u_h\| 	 \leq C_2 h^{k+\ot},
\end{align*}
where we have used Lemma \ref{lem:pstest} and Theorem \ref{thm:l2converge}.
Similarly, we have $\|u_x - (u_h)_x\|_{L^2(\partial \mathcal I_N)} \leq C_2 h^{k-\ot}$. By the property of $L^2$ projection $h^{\frac{3}{2}}\|v_x - (v_h)_x\|_{L^2(\partial \mathcal I_N)} +  h\|v_x-(v_h)_x\| + h^{\ot}\|v-v_h\|_{L^2(\partial \mathcal I_N)}+ \|v- v_h\| \leq C h^{k+1}\|v\|_{H^{k+1}(I)}$. Then it is straightforward that for scale invariant fluxes
\[
\| \widehat {v-v_h} \|_{L^2(\partial \mathcal I_N)}	\leq Ch^{k+\ot} \|v\|_{H^{k+1}(I)}, \quad \| \widetilde {(v-v_h)_x} \|_{L^2(\partial \mathcal I_N)} \leq C h^{k - \ot} \|v\|_{H^{k+1}(I)} .
\]
Therefore, we have
\beq
\label{eqn:thetanest}
\abs{ \Theta_N} \leq C_2 h^{2k} \|v\|_{H^{k+1}(I)} .
\eeq

Combine the above three estimate and the fact $\|v\|_{H^{k+1}(I)} = \|\Phi\|_{H^{k+1}(I)}$, we have
$$
\|u(T_e) - u_h(T_e) \|_{H^{-(k+1)}(I)} \leq C_2 h^{2k}.
$$

Since we consider $u_h$ with optimal error estimates on uniform mesh with mesh size $h$, then the divided difference $d_h^\alpha u$ satisfies the linear \sch but with initial data $d_h^\alpha u_0, \alpha \leq k+1$ on shifted mesh. Similarly, $d_h^\alpha u_h$ also satisfies the DG scheme \eqref{eqn:scheme} but with shifted mesh and initial numerical discretization $d_h^\alpha u_h = \plt d_h^\alpha u_0$. Then by the same proof for $u-u_h$ above, 
\beq
\label{eqn:unegest}
\|d_h^{\alpha} ( u- u_h)(T_e)\|_{H^{-(k+1)}(I)}
 \leq C_{2+\alpha} h^{2k},
\eeq
where we used Taylor expansion to estimate $d_h^\alpha u$ to obtain the last inequality.

The following theorem is a result of \eqref{eqn:unegest} and Theorem  \ref{thm:bramble}.
\begin{Thm}
Let $u_h$ be the UWDG solution of \eqref{eqn:scheme}, suppose the conditions in Theorem  \ref{thm:bramble} and any of the assumptions A1/A2/A3 is satisfied, then on a uniform mesh
\beq
\label{eqn:postpest}
\|u(T_e) - u^* (T_e) \| \leq C_{k+3} h^{2k}.
\eeq
\end{Thm}

\section{Numerical Experiments}
\label{sec:numerical}

In this section, we provide numerical tests  demonstrating superconvergence properties.
In the proof, we see that the initial value of $u_h$ matters in estimating $\|u_h - u_I\|$, thus will impact the superconvergence estimation for $E_f$ and $E_{f_x}$. Therefore, in our numerical tests, we apply two types of initial discretization for $u_h$. For computing the postprocessed solution $u^*$, we use the standard $L^2$ projection $\plt u$ as numerical initialization to demonstrate the convergence enhancement ability of postprocessor. For verifying other superconvergence quantities, we apply the initial condition $u_h|_{t = 0} = u_I|_{t=0}$.
In order not to deteriorate the high order convergence rates, for temporal discretization, we use explicit Runge-Kutta fourth order method with $dt = c \cdot h^{2.5}$, $c = 0.05$ when $k=2$ and $c=0.01$ when $k=3,4$. 

\begin{Example}
\label{exa:ls3}
We compute \eqref{eqn:ls} on $[0, 2\pi]$ with exact solution $u(x,t) = exp(i3(x-3t))$ using DG scheme \eqref{eqn:scheme}. We verify the results with several flux parameters.
\end{Example}

In the following tables, we show the convergence rate for quantities $E_f, E_{f_x}, E_c, E_u, E_{u_x}, E_{u_{xx}}$ as defined in \eqref{eqn:efefxec} and \eqref{eqn:eusuperconv} as well as
\beq
E^* = \|u - u^*\|, \quad E_P = \|u_h - \pst u \|,
\eeq
which represent the error after postprocessing, and the error between numerical solution and the projected exact solution $\pst u$.
In addition, we test the superconvergence of the intermediate quantities $\zeta_h = \pst u -w - u_h$ as in Lemma \ref{lem:zetaxx2}, and introduce the following notations:
\[
E_{[\zeta_h]} = (\frac{1}{N} \sumj |[\zeta_h]|_{\jp}^2)^{\ot}, \quad E_{[(\zeta_h)_x]} = (\frac{1}{N} \sumj |[(\zeta_h)_x]|_{\jp}^2)^{\ot}.
\]

The numerical fluxes we tested include
\begin{enumerate}
\item  Tables \ref{tab:alt}, \ref{tab:altzeta}: A1 parameters, alternating flux, $ \ao = 0.5, \bo = \bt = 0,$ with nonuniform mesh;
\item  Tables \ref{tab:local}, \ref{tab:localzeta}: A1 parameters, a   scale invariant flux, $ \ao = 0.3, \bo = \frac{0.4}{h}, \bt = 0.4 h,$ with nonuniform mesh;
\item  Tables \ref{tab:central}, \ref{tab:centralzeta}: A2 parameters, central flux, $ \ao = \bo = \bt = 0,$ with uniform mesh;
\item  Tables \ref{tab:global}: A3 parameters, $\ao = 0.25,  \bt = 0, \bo = \frac{2}{h}, \frac{5}{h}, \frac{9}{h}$ for $k =2, 3, 4,$ respectively, with uniform mesh;
\item  Table \ref{tab:postperr}: all parameters mentioned above, with uniform mesh,
\end{enumerate}
where the nonuniform mesh is generated by perturbing the location of the nodes of a uniform mesh by $10\%$ of mesh size.

We first verify the results in Theorems \ref{thm:efec}, \ref{thm:uhprojdiff} by examining Tables \ref{tab:alt}, \ref{tab:local}, \ref{tab:central}, \ref{tab:global}, where the parameters satisfy assumption A1, A1, A2, A3, respectively.
We observe that the scheme can achieve at least the theoretical order of convergence for the quantities in these two theorems. 
To be more specific, $E_P$ shows $(k+\min(3,k))$-th order of convergence. $E_u, E_{u_x}, E_{u_{xx}}$ are shown to have $(k+2)$-th, $(k+1)$-th and $k$-th order of convergence
, respectively. Note that when $k=2$, in Tables \ref{tab:local} and \ref{tab:global}, 
there are situations when no superconvergence point exists. This finding
shows an evidence to the assertion that $D^{s}$ defined in \eqref{eqn:dsj}
could be empty sets. The order of convergence for $E_f, (E_f)_x, E_c$ in all tables are $2k$. In addition, Table \ref{tab:global}
shows that when $k$ is even and assumption A3 is satisfied, the convergence order for all quantities are the same as when any of assumption A1/A2 is satisfied, which is one order higher than
the estimates in Theorems \ref{thm:efec}, \ref{thm:uhprojdiff}.
In Tables \ref{tab:alt} and \ref{tab:local}, we used nonuniform mesh in numerical test. The quantities tested have
similar order of convergence compared to the order of convergence on uniform mesh.
Another interesting observation is the order of convergence of $E_{f_x}$. Our numerical tests show that $E_{f_x}$
converges at an order of $2k$ for all four sets of parameters, which is at least one order
higher than the estimates in Theorem \ref{thm:efec}.


Next, we test the order of convergence for quantities in Lemma \ref{lem:zetaxx2}. 
In Tables \ref{tab:altzeta} and \ref{tab:localzeta}, we observe clean convergence order of $2k-1$, $2k+1$, $2k$ for 
$\|(\zeta_h)_{xx}\|$, $E_{[\zeta_h]}$, $E_{[(\zeta_h)_x]}$ when $k$ is even and $2k$, $2k+2$, $2k+1$ for these three quantities when $k$ is odd. 
In Table \ref{tab:centralzeta}, the order of convergence has some fluctuation, but the quantities are shown to have the same order of convergence as those in Tables \ref{tab:altzeta} and \ref{tab:localzeta}.
These convergence rates are consistent with the results in Lemma \ref{lem:zetaxx2}.


Lastly, we test the order of convergence for $E^*$ on uniform mesh for the four sets of parameters. 
Table \ref{tab:postperr} shows that $E^*$ has a convergence rate of at least $2k$, and can go up to $2k+2$. 
Similar higher order of convergence behaviors   exists in the literature \cite{cockburn2003enhanced, ryan2005extension}.

\begin{table}[!h]
    \centering
    \tiny
	\caption{Example \ref{exa:ls3}. Error table when using alternating flux on nonuniform mesh. Ending time $T_e=1$, $x \in [0, 2\pi]$.}
	\label{tab:alt}
	\hspace*{-2.4cm}
	\begin{tabular}{|c|c|c|c|c|c|c|c|c|c|c|c|c|c|c|c|c|c|}
	\hline
	& N & $L^2$ error & order & 	$E_P$ & order & $E_{u_{xx}}$& order & $E_{u_x}$& order & $E_u$ 	& order & $E_f$	 & order &
$E_{f_x}$   & order & $E_c$ & order	\\
	\hline
	\multirow{6}{1em}{$ P^2$} 
&  10&    2.68E-01&   -&    2.53E-01&   -&    2.36E+00&   -&    8.92E-01&   -&    2.83E-01&   -&    2.79E-01&   -&    9.10E-01&   -&    5.00E-01&   -\\
&  20&    2.68E-02&   3.32&    2.47E-02&   3.36&    2.78E-01&   3.09&    7.43E-02&   3.59&    2.30E-02&   3.62&    2.31E-02&   3.60&    6.94E-02&   3.71&    4.48E-02&   3.48\\
&  40&    2.00E-03&   3.75&    1.42E-03&   4.12&    6.02E-02&   2.21&    5.29E-03&   3.81&    1.47E-03&   3.97&    1.46E-03&   3.98&    4.45E-03&   3.96&    2.91E-03&   3.94\\
&  80&    1.91E-04&   3.39&    9.22E-05&   3.95&    1.50E-02&   2.00&    3.83E-04&   3.79&    9.12E-05&   4.01&    9.11E-05&   4.01&    2.75E-04&   4.02&    1.82E-04&   4.00\\
& 160&    2.19E-05&   3.12&    5.83E-06&   3.98&    3.78E-03&   1.99&    3.24E-05&   3.56&    5.78E-06&   3.98&    5.77E-06&   3.98&    1.74E-05&   3.98&    1.15E-05&   3.98\\
	\hline
	\multirow{4}{1em}{$ P^3$} 
&  10&    1.02E-02&   -&    7.51E-03&   -&    1.60E-01&   -&    2.50E-02&   -&    6.81E-03&   -&    6.78E-03&   -&    1.98E-02&   -&    1.16E-02&   -\\
&  20&    5.65E-04&   4.18&    1.24E-04&   5.93&    2.00E-02&   3.00&    8.23E-04&   4.93&    1.40E-04&   5.61&    1.36E-04&   5.64&    4.08E-04&   5.60&    2.64E-04&   5.45\\
&  40&    2.94E-05&   4.26&    2.13E-06&   5.86&    2.43E-03&   3.04&    3.83E-05&   4.42&    2.16E-06&   6.01&    2.02E-06&   6.07&    6.06E-06&   6.07&    4.03E-06&   6.03\\
&  80&    1.87E-06&   3.98&    3.02E-08&   6.14&    2.99E-04&   3.02&    2.29E-06&   4.06&    3.73E-08&   5.86&    3.05E-08&   6.05&    9.13E-08&   6.05&    6.11E-08&   6.04\\
& 160&    1.18E-07&   3.98&    4.57E-10&   6.05&    3.76E-05&   2.99&    1.45E-07&   3.98&    8.04E-10&   5.53&    4.59E-10&   6.05&    1.38E-09&   6.05&    9.25E-10&   6.05\\

	\hline
	\multirow{6}{1em}{$P^4$}	
&  10&    6.45E-04&   -&    8.76E-05&   -&    1.79E-02&   -&    8.67E-04&   -&    1.09E-04&   -&    9.69E-05&   -&    2.83E-04&   -&    1.62E-04&   -\\
&  20&    2.06E-05&   4.97&    6.12E-07&   7.16&    1.25E-03&   3.84&    3.06E-05&   4.82&    1.04E-06&   6.72&    5.02E-07&   7.59&    1.57E-06&   7.49&    9.69E-07&   7.39\\
&  40&    6.83E-07&   4.91&    2.52E-09&   7.92&    7.45E-05&   4.06&    9.00E-07&   5.09&    1.28E-08&   6.34&    1.76E-09&   8.16&    5.26E-09&   8.22&    3.50E-09&   8.11\\
&  80&    2.04E-08&   5.07&    1.55E-11&   7.35&    4.70E-06&   3.99&    2.78E-08&   5.02&    1.94E-10&   6.04&    6.99E-12&   7.97&    2.09E-11&   7.97&    1.40E-11&   7.97\\
& 160&    6.10E-10&   5.06&    1.02E-13&   7.24&    2.85E-07&   4.04&    8.34E-10&   5.06&    2.85E-12&   6.09&    2.79E-14&   7.97&    5.55E-13&   5.23&    5.07E-14&   8.10\\
	\hline
 	\end{tabular}
\end{table}

\begin{table}[!h]
\centering
\tiny
\tiny
	\caption{Example \ref{exa:ls3}. Error table for intermediate quantities when using alternating flux on nonuniform mesh. Ending time $T_e=1$, $x \in [0, 2\pi]$.}
	\label{tab:altzeta}
	\begin{tabular}{|c|c|c|c|c|c|c|c|c|c|c|}
	\hline
	& N & $\|\zeta_h\|$ error & order &  $\|(\zeta_h)_{xx}\|$ & order & $E_{[\zeta_h]}$& order & $E_{[(\zeta_h)_x]}$ 	& order \\
	\hline
	\multirow{6}{1em}{$ P^2$} 
&  10&    3.96E-01&   -&    3.02E+00&   -&    4.51E-02&   -&    3.37E-01&   -\\
&  20&    3.28E-02&   3.60&    2.23E-01&   3.76&    1.57E-03&   4.84&    1.82E-02&   4.21\\
&  40&    2.08E-03&   3.98&    1.42E-02&   3.98&    4.21E-05&   5.22&    9.42E-04&   4.28\\
&  80&    1.29E-04&   4.01&    7.95E-04&   4.15&    1.17E-06&   5.16&    5.50E-05&   4.10\\
& 160&    8.17E-06&   3.98&    9.54E-05&   3.06&    3.83E-08&   4.94&    3.55E-06&   3.96\\
	\hline		
	\multirow{4}{1em}{$ P^3$}
&  10&    9.54E-03&   -&    8.83E-02&   -&    7.52E-05&   -&    1.31E-03&   -\\
&  20&    1.93E-04&   5.63&    1.76E-03&   5.65&    1.75E-07&   8.75&    6.16E-06&   7.74\\
&  40&    2.86E-06&   6.08&    2.59E-05&   6.09&    3.42E-10&   9.00&    2.45E-08&   7.97\\
&  80&    4.31E-08&   6.05&    3.91E-07&   6.05&    1.45E-12&   7.88&    2.71E-10&   6.50\\
& 160&    6.87E-10&   5.97&    6.19E-09&   5.98&    6.76E-15&   7.74&    2.59E-12&   6.71\\
	\hline	
    	\multirow{6}{1 em}{$P^4$}
&  10&    1.35E-04&   -&    1.41E-03&   -&    4.97E-07&   -&    2.81E-05&   -\\
&  20&    7.10E-07&   7.57&    8.60E-06&   7.36&    1.20E-09&   8.69&    9.88E-08&   8.15\\
&  40&    2.50E-09&   8.15&    4.24E-08&   7.66&    2.56E-12&   8.88&    2.90E-10&   8.41\\
&  80&    9.90E-12&   7.98&    2.55E-10&   7.38&    2.98E-15&   9.74&    7.48E-13&   8.60\\
& 160&    3.58E-14&   8.11&    2.22E-12&   6.85&    8.90E-18&   8.39&    5.46E-15&   7.10\\
	\hline
 	\end{tabular}
\end{table}

\begin{table}[!h]
    \centering
    \tiny
	\caption{Example \ref{exa:ls3}. Error table when using flux parameters: $\alpha_1 = 0.3,  \beta_1 = \frac{0.4}{h},  \beta_2 = 0.4h $ on nonuniform mesh. Ending time $T_e=1$, $x \in [0, 2\pi]$.}
	\label{tab:local}
	\hspace*{-2.4cm}
	\begin{tabular}{|c|c|c|c|c|c|c|c|c|c|c|c|c|c|c|c|c|c|}
	\hline
	& N & $L^2$ error & order & 	$E_P$ & order & $E_{u_{xx}}$& order & $E_{u_x}$& order & $E_u$ 	& order & $E_f$	 & order &
$E_{f_x}$   & order & $E_c$ & order	\\
	\hline
\multirow{5}{1em}{$ P^2$}
&  40&    1.66E-02&   -&    1.28E-02&   -&    2.38E-01&   -&         DNE&   	-&    1.26E-02&   -&    1.26E-02&   -&    3.83E-02&   -&    2.49E-02&   -\\
&  80&    1.50E-03&   3.47&    7.46E-04&   4.10&    2.52E-02&   3.24&         DNE&    	-&    7.43E-04&   4.08&    7.42E-04&   4.08&    2.28E-03&   4.07&    1.48E-03&   4.07\\
& 160&    1.70E-04&   3.14&    4.91E-05&   3.92&    4.51E-03&   2.48&         DNE&    	-&    4.82E-05&   3.95&    4.82E-05&   3.95&    1.47E-04&   3.95&    9.63E-05&   3.94\\
& 320&    2.16E-05&   2.98&    3.12E-06&   3.98&    8.89E-04&   2.34&         DNE&    	-&    3.09E-06&   3.96&    3.09E-06&   3.96&    9.39E-06&   3.97&    6.18E-06&   3.96\\
& 640&    2.64E-06&   3.03&    1.96E-07&   3.99&    2.00E-04&   2.15&         DNE&    	-&    1.94E-07&   3.99&    1.94E-07&   3.99&    5.86E-07&   4.00&    3.88E-07&   3.99\\
	\hline	
	\multirow{5}{1em}{$ P^3$}

&  10&    2.08E-02&   -&    1.57E-02&   -&    2.00E-01&   -&    4.47E-02&   -&    1.44E-02&   -&    1.41E-02&   -&    4.38E-02&   -&    2.49E-02&   -\\
&  20&    1.14E-03&   4.19&    2.75E-04&   5.84&    2.23E-02&   3.17&    1.42E-03&   4.98&    3.08E-04&   5.54&    3.05E-04&   5.53&    9.17E-04&   5.58&    5.92E-04&   5.39\\
&  40&    5.91E-05&   4.27&    4.83E-06&   5.83&    2.72E-03&   3.03&    6.57E-05&   4.43&    4.70E-06&   6.04&    4.56E-06&   6.06&    1.37E-05&   6.06&    9.12E-06&   6.02\\
&  80&    3.75E-06&   3.98&    6.81E-08&   6.15&    3.36E-04&   3.02&    3.91E-06&   4.07&    7.56E-08&   5.96&    6.88E-08&   6.05&    2.06E-07&   6.05&    1.38E-07&   6.04\\
& 160&    2.39E-07&   3.97&    1.09E-09&   5.96&    4.23E-05&   2.99&    2.47E-07&   3.98&    1.03E-09&   6.20&    1.10E-09&   5.97&    3.29E-09&   5.97&    2.21E-09&   5.97\\
	\hline
	\multirow{5}{1em}{$P^4$}

&  10&    9.72E-04&   -&    1.48E-04&   -&    1.93E-02&   -&    1.32E-03&   -&    1.77E-04&   -&    1.65E-04&   -&    4.69E-04&   -&    2.73E-04&   -\\
&  20&    3.17E-05&   4.94&    1.04E-06&   7.16&    1.37E-03&   3.82&    4.63E-05&   4.83&    1.56E-06&   6.83&    8.66E-07&   7.57&    2.52E-06&   7.54&    1.66E-06&   7.36\\
&  40&    1.05E-06&   4.91&    4.15E-09&   7.97&    8.18E-05&   4.06&    1.35E-06&   5.10&    1.83E-08&   6.41&    3.04E-09&   8.15&    8.85E-09&   8.15&    6.02E-09&   8.11\\
&  80&    3.14E-08&   5.07&    2.51E-11&   7.37&    5.16E-06&   3.99&    4.27E-08&   4.99&    2.82E-10&   6.02&    1.21E-11&   7.98&    3.76E-11&   7.88&    2.41E-11&   7.97\\
& 160&    9.39E-10&   5.06&    1.63E-13&   7.27&    3.13E-07&   4.04&    1.27E-09&   5.07&    4.12E-12&   6.10&    4.41E-14&   8.10&    1.30E-13&   8.18&    8.73E-14&   8.11\\
	\hline	
 	\end{tabular}
\end{table}

\begin{table}[!h]
\centering
\tiny
\tiny
	\caption{Example \ref{exa:ls3}. Error table for intermediate quantities when using flux parameters: $\alpha_1 = 0.3,  \beta_1 = \frac{0.4}{h},  \beta_2 = 0.4h$ on nonuniform mesh. Ending time $T_e=1$, $x \in [0, 2\pi]$.}
	\label{tab:localzeta}
	\begin{tabular}{|c|c|c|c|c|c|c|c|c|c|c|}
	\hline
	& N & $\|\zeta_h\|$ error & order &  $\|(\zeta_h)_{xx}\|$ & order & $E_{[\zeta_h]}$& order & $E_{[(\zeta_h)_x]}$ 	& order \\
	\hline
	\multirow{5}{1em}{$ P^2$} 
&  40&    1.46E-02&   -&    2.67E-01&   -&    2.55E-03&   -&    2.13E-02&   -\\
&  80&    9.35E-04&   3.97&    2.57E-02&   3.38&    7.74E-05&   5.04&    1.25E-03&   4.09\\
& 160&    5.96E-05&   3.97&    2.86E-03&   3.17&    2.52E-06&   4.94&    7.56E-05&   4.05\\
& 320&    3.76E-06&   3.99&    3.30E-04&   3.11&    7.74E-08&   5.02&    4.76E-06&   3.99\\
& 640&    2.38E-07&   3.98&    4.25E-05&   2.96&    2.57E-09&   4.91&    3.19E-07&   3.90\\ 
	\hline	
	\multirow{5}{1em}{$ P^3$}
&  10&    2.02E-02&   -&    1.78E-01&   -&    5.73E-04&   -&    1.15E-03&   -\\
&  20&    4.31E-04&   5.55&    3.90E-03&   5.52&    1.09E-06&   9.04&    5.48E-06&   7.71\\
&  40&    6.46E-06&   6.06&    5.83E-05&   6.06&    3.27E-09&   8.38&    2.99E-08&   7.52\\
&  80&    9.72E-08&   6.05&    8.76E-07&   6.06&    9.55E-12&   8.42&    1.84E-10&   7.34\\
& 160&    1.55E-09&   5.97&    1.40E-08&   5.97&    4.28E-14&   7.80&    1.56E-12&   6.89\\
	\hline	
	\multirow{5}{1em}{$P^4$}
&  10&    2.27E-04&   -&    2.23E-03&   -&    1.03E-06&   -&    2.06E-06&   -\\
&  20&    1.22E-06&   7.54&    1.30E-05&   7.42&    4.13E-09&   7.96&    1.66E-08&   6.96\\
&  40&    4.30E-09&   8.15&    5.93E-08&   7.78&    1.01E-11&   8.67&    6.95E-11&   7.90\\
&  80&    1.71E-11&   7.98&    4.66E-10&   6.99&    1.82E-14&   9.12&    2.33E-13&   8.22\\
& 160&    6.17E-14&   8.11&    3.19E-12&   7.19&    3.61E-17&   8.98&    8.34E-16&   8.13\\
	\hline
 	\end{tabular}
\end{table}

\begin{table}[!h]
    \centering
    \tiny
	\caption{Example \ref{exa:ls3}. Error table when using central flux on uniform mesh. Ending time $T_e=1$, $x \in [0, 2\pi]$.}
	\label{tab:central}
	\hspace*{-2.3cm}
	\begin{tabular}{|c|c|c|c|c|c|c|c|c|c|c|c|c|c|c|c|c|c|}
	\hline
	& N & $L^2$ error & order & 	$E_P$ & order & $E_{u_{xx}}$& order  & $E_{u_x}$& order & $E_u$ 	& order & $E_f$	 & order &
$E_{f_x}$   & order & $E_c$ & order\\
	\hline
	\multirow{6}{1em}{$ P^2$} 
&  40&    4.20E-03&   -&    3.21E-03&   -&    4.86E-01&   -&    3.39E-02&   -&    3.24E-03&   -&    3.21E-03&   -&    9.58E-03&   -&    6.36E-03&   -\\
&  80&    4.31E-04&   3.29&    2.23E-04&   3.85&    1.33E-01&   1.87&    4.49E-03&   2.92&    2.25E-04&   3.85&    2.23E-04&   3.85&    6.86E-04&   3.80&    4.44E-04&   3.84\\
& 160&    4.92E-05&   3.13&    1.43E-05&   3.96&    3.41E-02&   1.97&    5.69E-04&   2.98&    1.45E-05&   3.96&    1.43E-05&   3.96&    3.90E-05&   4.14&    2.86E-05&   3.96\\
& 320&    5.99E-06&   3.04&    9.01E-07&   3.99&    8.57E-03&   1.99&    7.14E-05&   2.99&    9.10E-07&   3.99&    9.01E-07&   3.99&    3.00E-06&   3.70&    1.80E-06&   3.99\\
& 640&    7.44E-07&   3.01&    5.60E-08&   4.01&    2.15E-03&   2.00&    8.94E-06&   3.00&    5.66E-08&   4.01&    5.60E-08&   4.01&    1.51E-07&   4.31&    1.12E-07&   4.01\\
	\hline	
	\multirow{6}{1em}{$ P^3$} 	
&  20&    3.18E-04&   -&    7.32E-05&   -&    4.31E-02&   -&    3.34E-03&   -&    1.88E-04&   -&    7.28E-05&   -&    2.16E-04&   -&    1.41E-04&   -\\
&  40&    1.71E-05&   4.21&    1.02E-06&   6.16&    5.49E-03&   2.97&    2.04E-04&   4.03&    5.17E-06&   5.19&    1.02E-06&   6.16&    3.07E-06&   6.14&    2.03E-06&   6.12\\
&  80&    1.03E-06&   4.05&    1.55E-08&   6.04&    6.89E-04&   2.99&    1.27E-05&   4.01&    1.63E-07&   4.99&    1.54E-08&   6.04&    4.63E-08&   6.05&    3.10E-08&   6.03\\
& 160&    6.41E-08&   4.01&    2.41E-10&   6.01&    8.62E-05&   3.00&    7.91E-07&   4.00&    5.04E-09&   5.01&    2.39E-10&   6.01&    7.18E-10&   6.01&    4.81E-10&   6.01\\
& 320&    4.00E-09&   4.00&    3.76E-12&   6.00&    1.08E-05&   3.00&    4.94E-08&   4.00&    1.57E-10&   5.01&    3.73E-12&   6.00&    1.12E-11&   6.00&    7.51E-12&   6.00\\
	\hline
	\multirow{6}{1em}{$ P^4$} 
&  10&    5.04E-04&   -&    7.75E-05&   -&    1.14E-01&   -&    1.32E-02&   -&    2.21E-04&   -&    7.63E-05&   -&    2.15E-04&   -&    1.27E-04&   -\\
&  20&    2.10E-05&   4.58&    4.91E-07&   7.30&    1.11E-02&   3.36&    6.17E-04&   4.42&    5.70E-06&   5.28&    4.52E-07&   7.40&    1.26E-06&   7.42&    8.66E-07&   7.20\\
&  40&    7.32E-07&   4.84&    2.65E-09&   7.53&    8.01E-04&   3.79&    2.17E-05&   4.83&    1.05E-07&   5.76&    2.05E-09&   7.78&    6.17E-09&   7.67&    4.04E-09&   7.74\\
&  80&    2.36E-08&   4.96&    1.60E-11&   7.37&    5.21E-05&   3.94&    6.76E-07&   5.01&    1.72E-09&   5.93&    8.31E-12&   7.94&    2.57E-11&   7.91&    1.66E-11&   7.93\\
& 160&    7.42E-10&   4.99&    1.13E-13&   7.15&    3.29E-06&   3.99&    2.19E-08&   4.95&    2.72E-11&   5.98&    3.27E-14&   7.99&    9.96E-14&   8.01&    6.55E-14&   7.98\\
	\hline	
 	\end{tabular}
\end{table}

\begin{table}[!h]
\centering
\tiny
\tiny
	\caption{Example \ref{exa:ls3}. Error table for intermediate quantities when using central flux on uniform mesh. Ending time $T_e=1$, $x \in [0, 2\pi]$.}
	\label{tab:centralzeta}
	\begin{tabular}{|c|c|c|c|c|c|c|c|c|c|c|}
	\hline
	& N & $\|\zeta_h\|$ error & order &  $\|(\zeta_h)_{xx}\|$ & order & $E_{[\zeta_h]}$& order & $E_{[(\zeta_h)_x]}$ 	& order \\
	\hline
	\multirow{5}{1em}{$ P^2$} 
&  40&    4.53E-03&   -&    3.84E-02&   -&    3.04E-05&   -&    2.79E-04&   -\\
&  80&    3.15E-04&   3.85&    3.03E-03&   3.66&    1.16E-06&   4.71&    8.31E-06&   5.07\\
& 160&    2.02E-05&   3.96&    1.29E-04&   4.55&    1.79E-07&   2.70&    1.31E-06&   2.66\\
& 320&    1.27E-06&   3.99&    1.59E-05&   3.03&    7.14E-09&   4.65&    5.08E-08&   4.69\\
& 640&    7.92E-08&   4.01&    5.73E-07&   4.79&    2.90E-10&   4.62&    2.11E-09&   4.59\\
	\hline	
	\multirow{6}{1em}{$ P^3$}
&  20&    1.03E-04&   -&    9.27E-04&   -&    4.27E-08&  -&    5.25E-06&   -\\
&  40&    1.44E-06&   6.16&    1.29E-05&   6.17&    1.75E-10&   7.93&    4.32E-08&   6.93\\
&  80&    2.18E-08&   6.04&    1.99E-07&   6.02&    4.23E-13&   8.69&    1.64E-10&   8.04\\
& 160&    3.38E-10&   6.01&    3.05E-09&   6.03&    2.91E-16&  10.50&    3.24E-14&  12.31\\
& 320&    5.28E-12&   6.00&    4.76E-11&   6.00&    1.28E-18&   7.83&    1.41E-15&   4.52\\	\hline
	\multirow{6}{1em}{$ P^4$}
&  10&    1.06E-04&   -&    1.05E-03&   -&    7.67E-07&   -&    2.26E-05&   -\\
&  20&    6.37E-07&   7.37&    7.12E-06&   7.21&    3.03E-09&   7.99&    7.70E-08&   8.20\\
&  40&    2.88E-09&   7.79&    3.20E-08&   7.80&    5.78E-12&   9.03&    6.36E-11&  10.24\\
&  80&    1.17E-11&   7.94&    1.84E-10&   7.44&    9.33E-15&   9.28&    5.24E-13&   6.92\\
& 160&    4.63E-14&   7.98&    2.45E-12&   6.23&    3.04E-17&   8.26&    8.82E-16&   9.21\\
	\hline	
 	\end{tabular}
\end{table}

\begin{table}[!h]
    \centering
    \tiny
	\caption{Example \ref{exa:ls3}. Error table when using flux parameters: $\alpha_1 = 0.25, \beta_1 =  \frac{2}{h}, \frac{5}{h}, \frac{9}{h}, \bt = 0$ on uniform mesh. Ending time $T_e=1$, $x \in [0, 2\pi]$.}
	\label{tab:global}
	\hspace*{-2.4cm}
	\begin{tabular}{|c|c|c|c|c|c|c|c|c|c|c|c|c|c|c|c|c|c|}
	\hline
	& N & $L^2$ error & order & 	$E_P$ & order & $E_{u_{xx}}$& order & $E_{u_x}$& order & $E_u$ 	& order & $E_f$	 & order &
$E_{f_x}$   & order & $E_c$ & order	\\
	\hline
	\multirow{6}{1em}{$ P^2$} 	
&  80&    1.41E-03&   -&   		8.17E-05&   	-&    DNE&    -&    1.15E-02&   -&    1.19E-04&   -&    8.07E-05&   -&    1.71E-04&   -&    1.61E-04&   -\\
& 160&    1.65E-04&   3.09&    4.74E-06&   4.11&    DNE&    -&    1.34E-03&   3.11&    6.89E-06&   4.11&    4.67E-06&   4.11&    5.39E-06&   4.99&    9.34E-06&   4.11\\
& 320&    2.03E-05&   3.02&    2.92E-07&   4.02&    DNE&    -&    1.65E-04&   3.02&    4.21E-07&   4.03&    2.86E-07&   4.03&    1.82E-06&   1.57&    5.75E-07&   4.02\\
& 640&    2.53E-06&   3.01&    1.80E-08&   4.03&    DNE&    -&    2.05E-05&   3.01&    2.62E-08&   4.01&    1.78E-08&   4.00&    1.40E-07&   3.70&    3.58E-08&   4.01\\
&1280&    3.16E-07&   3.00&    1.22E-09&   3.88&    DNE&    -&    2.55E-06&   3.01&    1.71E-09&   3.94&    1.21E-09&   3.88&    5.55E-09&   4.65&    2.43E-09&   3.88\\
	\hline	
	\multirow{6}{1em}{$ P^3$} 
&  20&    8.27E-04&   -&    4.58E-05&   -&    2.98E-01&   -&    6.63E-03&   -&    3.40E-04&   -&    3.50E-05&   -&    9.99E-05&   -&    7.80E-05&   -\\
&  40&    3.92E-05&   4.40&    5.20E-07&   6.46&    3.11E-02&   3.26&    3.60E-04&   4.20&    8.78E-06&   5.27&    4.26E-07&   6.36&    1.24E-06&   6.34&    9.58E-07&   6.35\\
&  80&    2.29E-06&   4.10&    7.54E-09&   6.11&    3.72E-03&   3.06&    2.18E-05&   4.05&    2.61E-07&   5.07&    6.24E-09&   6.09&    1.87E-08&   6.05&    1.42E-08&   6.08\\
& 160&    1.40E-07&   4.03&    1.16E-10&   6.03&    4.60E-04&   3.02&    1.35E-06&   4.01&    8.05E-09&   5.02&    9.60E-11&   6.02&    2.86E-10&   6.03&    2.19E-10&   6.02\\
& 320&    8.74E-09&   4.01&    1.80E-12&   6.01&    5.74E-05&   3.00&    8.43E-08&   4.00&    2.50E-10&   5.01&    1.49E-12&   6.01&    4.47E-12&   6.00&    3.41E-12&   6.00\\
	\hline
	\multirow{6}{1em}{$ P^4$} 
&  20&    5.10E-04&   -&    2.08E-04&   -&    3.76E-01&   -&    3.96E-03&   -&    1.36E-04&   -&    1.08E-05&   -&    3.52E-05&   -&    2.10E-05&   -\\
&  40&    8.28E-06&   5.95&    2.38E-07&   9.77&    1.24E-02&   4.92&    6.76E-05&   5.87&    1.16E-06&   6.87&    2.19E-08&   8.95&    6.74E-08&   9.03&    4.34E-08&   8.92\\
&  80&    1.87E-07&   5.47&    1.04E-09&   7.84&    5.64E-04&   4.47&    1.55E-06&   5.45&    1.33E-08&   6.45&    6.23E-11&   8.46&    2.15E-10&   8.29&    1.25E-10&   8.45\\
& 160&    5.44E-09&   5.10&    7.11E-12&   7.19&    3.29E-05&   4.10&    4.53E-08&   5.09&    1.94E-10&   6.09&    2.28E-13&   8.09&    6.15E-13&   8.45&    4.56E-13&   8.09\\
	\hline	
 	\end{tabular}
\end{table}

\begin{table}[!h]
\centering
\tiny
\caption{Example \ref{exa:ls3}. Postprocessing error table for the four sets of parameters. Ending time $T_e=1$,  uniform mesh on $x \in [0, 2\pi]$. The first row below labels the parameters by $(\tilde \ao, \tilde \bo, \tilde \bt)$.}
\label{tab:postperr}
	\begin{tabular}{|c|c|c|c|c|c|c|c|c|c|c|}
	\hline
\multicolumn{2}{|c|}{Fluxes}	& \multicolumn{2}{|c|}{(0.5,0,0)} & \multicolumn{2}{|c|}{(0, 0, 0)} &  \multicolumn{2}{|c|}{(0.3, 0.4, 0.4)} & \multicolumn{2}{|c|}{(0.25, \{2, 5, 9\}, 0)}	\\	\hline
	& N	& $E^*$		& order	& $E^*$		& order	& $E^*$		& order 	& $E^*$ 	& order	\\
\hline
	\multirow{5}{1em}{$ P^2$} 	
& 10		&    1.00E+00&   -			&    2.81E-01&   -			&    1.00E+00&   -		&    1.53E-01&   -\\	
& 20		&    2.84E-01&   1.81         	&    3.71E-02&   2.92              	 &    1.20E-01&   3.06	&    8.05E-02&   0.93\\                                                                                                                         
& 40		&    2.11E-02&   3.75          	&    3.23E-03&   3.52             	&    9.63E-03&   3.64  	&    2.68E-03&   4.91\\                                                                                                                        
& 80		&    1.37E-03&   3.94           	&    2.24E-04&   3.85                 	&    7.55E-04&   3.67  	&   1.20E-04&   4.49\\                                                                                                                     
& 160		&    8.69E-05&   3.98		&    1.44E-05&   3.96			&    5.13E-05&   3.88		&   6.99E-06&   4.10\\
\hline
	\multirow{5}{1em}{$ P^3$} 	
& 10 		&    1.00E+00&   -		&    1.00E+00&   -	&    1.00E+00&   -		&    1.00E+00&   -\\
& 20 		&    6.04E-02&   4.05			&    6.29E-02&   3.99		&    6.05E-02&   4.05		&    7.02E-02&   3.83\\
& 40 		&    5.39E-04&   6.81			&    6.05E-04&   6.70		&    5.26E-04&   6.85		&    5.46E-04&   7.01\\
& 80 		&    3.28E-06&   7.36                 &    5.04E-06&   6.91      	&    2.82E-06&   7.54  	&    2.91E-06&   7.55\\                                                                                                                                            
& 160 		&    3.14E-08&   6.70		&    6.49E-08&   6.28		&    2.04E-08&   7.11		&    1.79E-08&   7.34\\
\hline
	\multirow{5}{1em}{$ P^4$} 	
& 10 & 1.00E+00&   -				&    1.00E+00&   -		&    1.00E+00&   -		&    1.00E+00&   -\\
& 20 & 4.54E-02&   4.46				&    4.54E-02&   4.46			&    4.54E-02&   4.46		&    4.54E-02&   4.46\\
& 40 & 1.32E-04&   8.42				&    1.32E-04&   8.42			&    1.32E-04&   8.42		&    1.36E-04&   8.39\\
& 80 & 1.70E-07&   9.60				&    1.70E-07&   9.60			&    1.70E-07&   9.60		&    1.66E-07&   9.67\\
& 160 & 1.79E-10&   9.89				&    1.80E-10&   9.89			&    1.79E-10&   9.89		&    1.75E-10&   9.89\\
\hline
	\end{tabular}

\end{table}

\section{Conclusions and Future Work}
\label{sec:conclusion}
 
In this paper, we studied the superconvergence property of the
UWDG methods with scale invariant fluxes for linear \sch equation
 in one dimension with periodic boundary condition. When
 $k$ is odd, and $k$ is even with the flux parameters satisfying certain assumptions, we  proved
$(2k)$-th order convergence rate for cell averages and numerical flux, and $(2k-1)$-th order convergence
rate for numerical flux for the  derivative. In addition,   the numerical solution
is convergent towards a special projection with $(k+\min(3,k))$-th order convergence rate. The results were obtained by the correction
function techniques in \cite{cao2014superconvergence} and intermediate results of the superconvergence
of the second derivative and jump across cell interfaces of the difference between numerical solution
and projected exact solution. However, for some special flux parameter choices when $k$ is even, 
such intermediate results are no longer valid. Therefore, under this condition, the provable convergence rate is one order lower than the previous cases, though numerical results seem to suggest otherwise. Indeed, our numerical experiments show
 $(2k)$-th convergence rate for cell average, numerical flux, and the numerical flux for the derivative, 
 and $(k+\min(3,k))$-th convergence rate of the difference between numerical solution and a special projection.
 The surprising finding that the numerical flux of solution and derivative of solution both converge at
 rate $2k$ indicates that our proof can be improved for the numerical
 flux for derivative.
 We also showed that the  convergence order of UWDG scheme can 
 be enhanced to $2k$ by postprocessing. In numerical tests, the orders of convergence for the postprocessed
 solution are at least $2k$, and can go up to $2k+2$ when $k \geq 3$.

 There are some recent development of extending the correction function technique to nonlinear equations and high dimensional equations.
 It would be our future work to extend this work to nonlinear \sch equations in two dimensional setting. Also, there
 are developments in the negative norm estimates for the error of DG schemes for nonlinear equations, it will also be an interesting subject to consider.


\bibliographystyle{abbrv}
\bibliography{ac@msu} 

\appendix
\section{Appendix}
\subsection{Collections of intermediate results}
\label{apdx:localerr}
In this section, we list some results that will be used in the rest of the appendix.
First, we gather some results from \cite{2018arXiv180105875C}.  
(59) - (63) in \cite{2018arXiv180105875C} yields 
\beq
\label{eqn:minv1}
\sumj \|r_j\|_\infty \leq C, \quad \text{if A2},
\eeq
where $r_j$ has been defined in \eqref{eqn:rj}.
(59) and (71) and the first equation in A.3.3 in \cite{2018arXiv180105875C} yields 
\beq
\label{eqn:minv2}
\sumj \|r_j\|_\infty \leq Ch^{-2}, \quad \text{if A3}.
\eeq

Next, we provide estimates of the Legendre coefficients in neighboring cells of equal size. 
 
If $u \in W^{k+2+n,\infty}(I)$, then expand $\hat u_j(\xi)$ at $\xi = -1$ in \eqref{eqn:ltcoefcompute} 
by Taylor series, we have for $m \geq k+1$, $\exists z\in [-1,1]$, s.t.
\beq
\label{eqn:ujmtaylor}
	\begin{split}
		u_{j,m} & = C \int_{-1}^1 \frac{d}{d\xi^{k+1}} \Big(\sum_{s=0}^n \frac{d}{d\xi^s}\hat u_j (-1) \frac{(\xi+1)^s}{s!}+ 
	\frac{d}{d\xi^{n+1}} \hat u_j(z) \frac{(\xi+1)^{n+1}}{(n+1)!} \Big)  \frac{d}{d\xi^{m-k-1}} (\xi^2 -1)^{m} d\xi,	\\
				& = \sum_{s=0}^n \theta_s h_j^{k+1+s} u^{(k+1+s)}(x_{\jm}) + O (h_j^{k+2+n} |u|_{W^{k+2+n,\infty}(I_j)}),
	\end{split}
\eeq
where $\theta_l$ are constants   independent of $u$ and $h_j$.

Therefore, when $h_j = h_{j+1}$, we use Taylor expansion again, and compute the difference of two $u_{j, m}$ from neighboring cells
\beq
\label{eqn:ujmdiff}
\begin{split}
u_{j, m} - u_{j+1,m} & = \sum_{s=1}^n \mu_s h_j^{k+1+s} u^{(k+1+s)}(x_{\jm}) + O (h_j^{k+2+n} |u|_{W^{k+2+n,\infty}(I_j \cup I_{j+1})}).	\\
\end{split}
\eeq 
Then we obtain the estimates
\beq
\label{eqn:ujmdiffest}
|u_{j, m} - u_{j+1,m} + \sum_{s=1}^n \mu_s h_j^{k+1+s} u^{(k+1+s)}(x_{\jm}) | 	\leq C h^{k+2+n}|u|_{W^{k+2+n,\infty}(I_j \cup I_{j+1})},
\eeq
where $\mu_s$ are constants independent of $u$ and $h_j$.

\subsubsection{Two convolution-like operators}
In the proof of Lemmas 3.8 and 3.9 in \cite{2018arXiv180105875C}, we used Fourier analysis for error analysis. Now we extract the main ideas and generalize the 
results to facilitate the proof of superconvergence results in Lemmas \ref{lem:pdagprop} and  \ref{lem:west}.

We define two operators on a periodic functions $u$ in $L^2(I)$:
\begin{subequations}
\begin{align}
	\label{eqn:sumladef1}
\sumla_{\lambda}u (x) & = \frac{1}{1-\lambda^N} \suml \lambda^l u(x+L \frac{l}{N}),	\\
	\label{eqn:sumladef2}
\sumnj u(x) & = \suml (-1)^l \frac{-N + 2l}{2} u(x+L \frac{l}{N}),
\end{align}
\end{subequations}
where $L = b-a$ is the size of $I$.

Expand $u$ by Fourier series, i.e., $u(x) = \sumn \hat f(n) e^{2\pi inx/L}$, we have
\[
\begin{split}
\sumla_{\lambda}u (x) 	& = \frac{1}{1-\lambda^N} \suml \lambda^l \sumn \hat f(n) e^{in (\frac{2\pi}{L} x + 2\pi \frac{l}{N})}  = \frac{1}{1-\lambda^N} \sumn \hat f(n) e^{\frac{2\pi}{L}  in x} \suml ( \lambda e^{i2\pi \frac{n}{N}} )^l	\\
			& = \sumn \frac{\hat f(n)}{1 - \lambda e^{2\pi i \frac{n}{N}}} e^{\frac{2\pi}{L}  inx},	\\
\sumnj u(x) 			& = \suml (-1)^l \frac{-N + 2l}{2} \sumn \hat f(n) e^{in (\frac{2\pi}{L} x + 2\pi \frac{l}{N})}	 = \sumn \hat f(n)  e^{\frac{2\pi}{L}  in x} \suml \frac{-N + 2l}{2} (-e^{i2\pi \frac{n}{N}})^l	\\
			& = \sumn \frac{-2e^{2\pi i \frac{n}{N}}}{(1 + e^{2\pi i \frac{n}{N}})^2}  \hat f(n) e^{\frac{2\pi}{L} in x}.
\end{split}
\]
In addition, we can apply the operator on the same function recursively, we have
\[
\begin{split}
	\sumla_{\lambda_1}^{\nu_1} \cdots \sumla_{\lambda_n}^{\nu_n} u (x)	& = \sumn \frac{1}{(1 - \lo e^{2\pi i \frac{n}{N}})^{\nu_1}} \cdots \frac{1}{(1 - \lambda_n e^{2\pi i \frac{n}{N}})^{\nu_n}}  \hat f(n) e^{i\frac{2\pi}{L}  inx},	\\
(\sumla_{\lambda})^\nu u (x) & = \sumn \frac{\hat f(n)}{(1 - \lambda e^{2\pi i \frac{n}{N}})^\nu} e^{\frac{2\pi}{L}  inx},	\\
(\sumnj)^\nu u(x) 	& = \sumn \Big(\frac{-2e^{2\pi i \frac{n}{N}}}{(1 + e^{2\pi i \frac{n}{N}})^2} \Big )^\nu \hat f(n) e^{\frac{2\pi}{L}  in x}.
\end{split}
\]

As shown in the proof of Lemmas 3.8 and   3.9 in \cite{2018arXiv180105875C}, if $\lambda_i, i \leq n$ is a complex number with $|\lambda_i| = 1$, independent of $h$, then
\beq
\label{eqn:sumlaest}
\sumla_{\lambda_1}^{\nu_1} \cdots \sumla_{\lambda_n}^{\nu_n} u (x) \leq C | u |_{W^{1+\sum_{i=1}^n \nu_i,1}(I)}, \quad (\sumnj u(x))^\nu \leq C | u |_{W^{1+2\nu,1}(I)}.
\eeq


\subsection{Proof of Lemma \ref{lem:mjm}}
\label{sec:mjmproof}
\begin{proof}
$$
A_j+B_j = G [L_{j,k-1}^-, L_{j,k}^-] + H [L_{j,k-1}^+, L_{j,k}^+]  = \ot \bmat 1 & 0 \\ 0 & \ohj \emat  M_+ + \bmat \ao & -\bt \\ -\bo & -\ao \emat \bmat 1 & 0 \\ 0 & \ohj \emat  M_-,
$$
where
\[
M_\pm = \bmat 1 & 0 \\ 0 & h_j \emat [L_{j,k-1}^- \pm L_{j,k-1}^+, L_{j,k}^- \pm L_{j,k}^+] = 
\bmat 1\pm(-1)^{k-1} & 1\pm(-1)^{k} \\  k(k-1)( 1\pm (-1)^k) & k(k+1)(1\pm(-1)^{k+1}) \emat.
\]

Therefore,
\[
(A_j+B_j)^{-1} =  \frac{1}{D_1}M_-^{-1} \bmat -\ao & \bt - \frac{h_j}{2k(k+(-1)^k)}	\\ \bo h_j - \frac{k(k - (-1)^k)}{2} & \ao h_j \emat
\]
where $D_1 = \frac{(-1)^kh_j}{2k(k + (-1)^k)} ( (-1)^k \g_j + \la_j)$ is bounded by definitions of  $\g_j, \la_j$ and mesh regularity condition.
Then
\[
\begin{split}
(A_j+B_j)^{-1} G \bmat 1 & 0 \\ 0 & \frac{1}{h_j} \emat & =  \frac{1}{D_1}M_-^{-1} \bmat -\ao & \tilde \bt h h_j^{-1} - \frac{1}{2k(k + (-1)^k)}	\\ \tilde \bo h^{-1} h_j - \frac{k(k -(-1)^k)}{2} & \ao \emat \bmat \ot + \ao & -\tilde \bt h h_j^{-1} \\ - \tilde \bo h^{-1} h_j & \ot - \ao \emat,	\\
(A_j+B_j)^{-1} H \bmat 1 & 0 \\ 0 & \frac{1}{h_j} \emat & =  \frac{1}{D_1}M_-^{-1} \bmat -\ao & \tilde \bt h h_j^{-1} - \frac{1}{2k(k + (-1)^k)}	\\ \tilde \bo h^{-1} h_j - \frac{k(k - (-1)^k)}{2} & \ao \emat \bmat \ot - \ao & \tilde \bt h h_j^{-1} \\ \tilde \bo h^{-1} h_j & \ot + \ao \emat
\end{split}
\]
and
\[
\mathcal{M}_{j, m} = (A_j+B_j)^{-1} G \bmat 1 & 0 \\ 0 & \frac{1}{h_j} \emat \bmat 1 \\ m(m+1) \emat +(-1)^m (A_j+B_j)^{-1} H \bmat 1 & 0 \\ 0 & \frac{1}{h_j} \emat \bmat 1 \\ -m(m+1) \emat.
\]

By mesh regularity condition, $\exists \sigma_1, \sigma_2, s.t., \sigma_1 h_j \leq h \leq \sigma_2 h_j$ and the proof is complete.

\end{proof}

\subsection{Proof of Lemma \ref{lem:pstest}}
\label{apdx:pstproof}
By Definition \ref{def:pst},
$\pst u  |_{I_j}= \sum_{m=0}^{k-2} u_{j,m} L_{j,m} + \acute{u}_{j,k-1}L_{j,k-1} + \acute{u}_{j,k}L_{j,k}.$ We solve the two coefficients $\acute{u}_{j,k-1}, \acute{u}_{j,k}$ on every cell $I_j$ according to definition \eqref{eqn:psteq}.

If assumption A1 is satisfied, it has been shown in Lemma 3.1 in \cite{2018arXiv180105875C} that \eqref{eqn:psteq} is equivalent to \eqref{eqn:psteq2}. Substitute $u$ and $u_x$ by \eqref{eqn:lt}, we obtain the following equation
\beq
\label{eqn:pdagflux2}
(A_j+B_j)
\begin{bmatrix}
\acute{u}_{j,k-1}	\\
\acute{u}_{j,k}
\end{bmatrix}
 = (A_j+B_j)
\begin{bmatrix}
{u}_{j,k-1}	\\
{u}_{j,k}
\end{bmatrix}
+
\sum_{m=k+1}^\infty u_{j,m} (G L_{j,m}^- + H L_{j,m}^+),
\eeq
the existence and uniqueness of the system above is ensured by assumption A1, that is, $\det (A_j+B_j) = 2(-1)^k \g_j \neq 0$. Thus, \eqref{eqn:pstcoefa1} is proven.

If any of the assumptions A2/A3 is satisfied, we obtain
\[
A
\bmat
\acute{u}_{j,k-1}	\\
\acute{u}_{j,k}
\emat
+ B
\bmat
\acute{u}_{j+1,k-1}	\\
\acute{u}_{j+1,k}
\emat
 = 
\sum_{m=k-1}^\infty u_{j,m}
G L_{m}^-
+ u_{j+1,m}
H L_{m}^+,
\]
which can be solved by a global linear system with coefficient matrix $M.$ The solution is
\[
\hspace{-0.8in}
\begin{split}
\bmat
\acute{u}_{j,k-1}	\\
\acute{u}_{j,k}
\emat
& = \suml r_l A^{-1} \Big ( A
\bmat
{u}_{j+l,k-1}	\\
{u}_{j+l,k}
\emat
+ B
\bmat
{u}_{j+l+1,k-1}	\\
{u}_{j+l+1,k}
\emat
+
\sum_{m=k+1}^\infty u_{j+l,m}
G L_{m}^-
+ u_{j+l+1,m}
H L_{m}^+
\Big ),		\\
& = \suml r_l \Big (\bmat
{u}_{j+l,k-1}	\\
{u}_{j+l,k}
\emat
- Q\bmat
{u}_{j+l+1,k-1}	\\
{u}_{j+l+1,k}
\emat
+ \sum_{m=k+1}^\infty u_{j+l,m}
[ L_{k-1}^- , L_k^- ]^{-1} L_{m}^-
- u_{j+l+1,m}
Q [ L_{k-1}^+ , L_k^+ ]^{-1} L_{m}^+ \Big )	\\
& = \bmat
{u}_{j,k-1}	\\
{u}_{j,k}
\emat
+ \sum_{m=k+1}^\infty \Big (\sum_{l=1}^{N-1} u_{j+l, m} V_{2,m} + u_{j, m} r_0 [ L_{k-1}^- , L_k^- ]^{-1} L_{m}^- - u_{j+N, m} r_N [ L_{k-1}^- , L_k^- ]^{-1} L_{m}^- \Big )	\\
& = 
\bmat
{u}_{j,k-1}	\\
{u}_{j,k}
\emat
+\sum_{m=k+1}^\infty \big (u_{j, m} V_{1,m} + \suml  u_{j+l,m} r_l V_{2,m}  \big ),
\end{split}
\]
where $r_N = Q^N (I_2 - Q^N)^{-1} =  r_0 - I_2 $ is used in the third equality. 
Therefore, \eqref{eqn:pstcoefa23} is proven. The proof of \eqref{eqn:pstest} is given in Lemmas 3.2, 3.4, 3.8, 3.9 in \cite{2018arXiv180105875C}. 

If $h_j = h_{j+1}$, denote 
\[
	\mathcal{U}_j = 
		\bmat
		\acute u_{j,k-1}	- u_{j,k-1}\\
		\acute u_{j,k}-u_{j,k}
		\emat
		-
		\bmat
		\acute u_{j+1,k-1} -u_{j+1,k-1}	\\
		\acute u_{j+1,k} -u_{j+1,k}
		\emat	
\]
then by \eqref{eqn:pstcoefa1} and \eqref{eqn:pstcoefa23}, we have
\[
\begin{split}
\mathcal{U}_j
& =
\sum_{m=k+1}^\infty (u_{j,m} - u_{j+1,m}) \mathcal{M}_{m}, \quad \text{if A1},	\\
\mathcal{U}_j
&  =
\sum_{m=k+1}^\infty \big ( (u_{j, m} - u_{j, m}) V_{1,m} + \suml  (u_{j+l,m} - u_{j+l+1,m}) r_l V_{2,m}  \big ), \  \text{if A2/A3}.
\end{split}
\]

When assumption A1 is satisfied \eqref{eqn:pstcoefdiff} is 
a direct result of \eqref{eqn:ujmdiff}.

When assumption A2 is satisfied, we have
\[
	\|\mathcal{U}_j\|_\infty \leq C (1+\suml \|r_l\|_{\infty} ) \max_{j+l\in \zn} h^{k+2} | u^{(k+2)}(x_{j+l-\ot}) | \leq Ch^{k+2} | u |_{W^{k+2,\infty}(I)},
\]
where \eqref{eqn:minv1}, \eqref{eqn:ujmdiff} and the fact that $V_{1,m}, V_{2,m}, \forall m \geq 0$ are constant matrices independent of $h$ are used in above inequalities.

When assumption A3 is satisfied, we perform more detailed computation of $\mathcal{S}_j$ and 
use Fourier analysis to bound it by utilizing the smoothness
and periodicity. If $u \in W^{k+2+n, \infty}(I)$, 
\[
	\begin{split}
	\mathcal{U}_j & = \suml r_l \sum_{m=k+1}^\infty  (u_{l+j,m} - u_{l+j+1,m}+ \sum_{s=1}^n \mu_s h_j^{k+1+s} u^{(k+1+s)}(x_{j+l-\ot}))V_{2,m}	\\
		& - \suml r_l \sum_{m=k+1}^\infty \sum_{s=1}^n \mu_s h_j^{k+1+s} u^{(k+1+s)}(x_{j+l-\ot}) V_{2,m} + \sum_{m=k+1}^\infty (u_{j, m} - u_{j, m}) V_{1,m}.
	\end{split}
\]

When $\frac{|\g|}{|\la|} < 1$, $Q=-A^{-1}B$ has two imaginary eigenvalues 
$\lo, \lt$ with $|\lo | = |\lt | =1$.
By (59) of \cite{2018arXiv180105875C}, we have $r_l= \frac{\lo^l}{1-\lo^N} Q_1 + \frac{\lt^l}{1-\lt^N} (I_2 - Q_1)$, where $Q_1$ is a constant matrix independent of $h$, and defined in (60) and (61) in \cite{2018arXiv180105875C}.

Thus, by \eqref{eqn:ujmdiff} and \eqref{eqn:ujmdiffest}
\[
	\begin{split}
	\| \mathcal{U}_j \|_\infty & \leq C \suml \|r_l\|_{\infty} h^{k+2+n} |u|_{W^{k+2+n,\infty}(I)}  + C h^{k+2}|u|_{W^{k+2,\infty}(I)}\\
		& + \| \sum_{m=k+1}^\infty  \sum_{s = 1}^n \mu_s h^{k+1+s} (Q_1 \sumla_{\lo} + (I_2-Q_1) \sumla_{\lt})u^{(k+1+s)}(x_{j-\ot}) V_{2,m} \|_\infty.
	\end{split}
\]

By \eqref{eqn:sumlaest}, we have
\[
	\| \mathcal{U}_j \|_\infty  \leq C_1 h^{k+2}.
\]

When $\frac{|\g|}{|\la|} =1$, $Q=-A^{-1}B$ has two repeated eigenvalues. By (71) of \cite{2018arXiv180105875C}, 
we have $r_l = \frac{(-1)^l}{2} I_2 + (-1)^l \frac{-N+2l}{4\g} Q_2$, where $Q_2/ \g$ is a constant matrix, 
we estimate $\mathcal{U}_j$ by the same procedure as previous case and obtain
\[
	\begin{split}
	\| \mathcal{U}_j \|_\infty & \leq C \suml \|r_l\|_{\infty} h^{k+2+n} |u|_{W^{k+4,\infty}(I)} + C h^{k+2}|u|_{W^{k+2,\infty}(I)}\\
		& + \| \sum_{m=k+1}^\infty  \sum_{s = 1}^n \mu_s h^{k+1+s} \ot (\sumla_{-1} + \frac{Q_2}{\g} \sumnj) u^{(k+1+n)}(x_{j-\ot})
				V_{2,m} \|_\infty \leq C_1 h^{k+2}.
	\end{split}
\]
Finally, the estimates for $\mathcal{U}_j$ is complete for all assumptions and \eqref{eqn:pstcoefdiff} is proven.

\subsection{Proof of Lemma \ref{lem:pdagprop}}
\label{apdx:pdagproof}
\begin{proof}
By the definition of $\pdag$, the solution of $\grave u_{j, k-1}, \grave u_{j, k}$ has similar linear algebraic system as \eqref{eqn:pstcoefa1}. That is, under assumption A2 or A3, the existence and uniqueness condition is $\det (A+B) = 2( (-1)^k \g + \la) \neq 0$. Thus,
\beq
\label{eqn:graveu}
\begin{bmatrix}
\grave{u}_{j,k-1}	\\
\grave{u}_{j,k}
\end{bmatrix}
=
\begin{bmatrix}
u_{j,k-1}	\\
u_{j,k}
\end{bmatrix}
+
\sum_{m=k+1}^\infty u_{j,m} \mathcal{M}_{m}.
\eeq
And then, by \eqref{eqn:mjm} and \eqref{eqn:ltcoef}, \eqref{eqn:pdagest} is proven.



If any of the assumptions A2/A3 is satisfied, then the difference can be written as
\[
Wu|_{I_j} = \pst u|_{I_j} - \pdag u|_{I_j}  = (\acute u_{j,k-1} - \grave u_{j,k-1} )L_{j,k-1} + (\acute u_{j,k} - \grave u_{j,k} ) L_{j,k}.
\]

The properties of $\pst u$ and $\pdag u$ yield the following coupled system
\[
A \bmat \acute u_{j,k-1} - \grave u_{j,k-1} \\ \acute u_{j,k} - \grave u_{j,k} \emat + B \bmat \acute u_{j+1,k-1} - \grave u_{j+1,k-1} \\ \acute u_{j+1,k} - \grave u_{j+1,k} \emat = \bmat {\tau}_{j}	\\ {\iota}_{j}	\emat, \quad \forall j \in \zn,
\]
$$
\label{eqn:wurhs}
\begin{bmatrix}
\tau_j\\
 \iota_j
\end{bmatrix}
 =
\bmat
u	\\
u_x
\emat \bigg \rvert_{x_{j + \ot}}
-
G
\bmat
\pdag u 	\\
(\pdag u)_x
\emat
\bigg \rvert_{x_{j+\ot}}^-
-
H
\bmat
\pdag u 	\\
(\pdag u)_x
\emat
\bigg \rvert_{x_{j+\ot}}^+
= 
G
\bmat
(u -\pdag u)\rvert_{x_{j+\ot}}^- - (u -\pdag u)\rvert_{x_{j+\frac{3}{2}}}^-	\\
(u -\pdag u)_x\rvert_{x_{j+\ot}}^- - (u -\pdag u)_x\rvert_{x_{j+\frac{3}{2}}}^-
\emat,
$$  
where the second equality was obtained by the definition of $\pdag u$ \eqref{eqn:pdagflux}.

Gather the relations above for all $j$   results in a large $2N \times 2N$ linear system with block circulant matrix $M,$ defined in \eqref{eqn:mdef}, as coefficient matrix, then the solution is
$$
\bmat \acute u_{j,k-1} - \grave u_{j,k-1} \\ \acute u_{j,k} - \grave u_{j,k} \emat  = \sum_{l=0}^{N-1} r_l A^{-1} \bmat {\tau}_{l+j}	\\ {\iota}_{l+j}	\emat ,  \quad j \in \zn,
$$
where by periodicity, when $l+j> N$, $\tau_{l+j} = \tau_{l+j-N}, \iota_{l+j} = \iota_{l+j-N}$.

On uniform mesh, by the definition of $R_{j,m}$ in \eqref{eqn:rjm}, $R_{j,m}(1)$ and $(R_{j,m})_x(1)$ are independent of $j$, we denote the corresponding values as $R_m(1)$ and $(R_m)_x(1)$ and let $R^-_{m} = [R_m(1), (R_m)_x(1)]^T$. By \eqref{eqn:updagdiff}, we have
$$
\bmat
(u -\pdag u)\rvert_{x_{j+\ot}}^- - (u -\pdag u)\rvert_{x_{j+\frac{3}{2}}}^-	\\
(u -\pdag u)_x\rvert_{x_{j+\ot}}^- - (u -\pdag u)_x\rvert_{x_{j+\frac{3}{2}}}^-
\emat
=
\sum_{m=k+1}^\infty (u_{j,m} - u_{j+1,m})
R^-_m
$$
and 
\beq
\label{eqn:a2}
\bmat \acute u_{j,k-1} - \grave u_{j,k-1} \\ \acute u_{j,k} - \grave u_{j,k} \emat = \sum_{l=0}^{N-1} r_l \big ( \sum_{m=k+1}^\infty (u_{l+j,m} - u_{l+j+1,m})
 A^{-1} 
G R^-_m
\big ) \doteq \mathcal{S}_j, \quad j \in \mathbb{Z}_N.
\eeq

We can estimate $\mathcal{S}_j$ by the same lines as the estimation of $\mathcal{U}_j$ in Appendix \ref{apdx:pstproof} and \eqref{eqn:projdiff} is proven.

\end{proof}

\subsection{Proof of Lemma \ref{lem:zetaxx}}
\label{apdx:zetaxxproof}
\begin{proof}
By error equation,  the symmetry of $A(\cdot, \cdot)$ and the definition of $s_h,$ we have
\beq
\label{eqa1}
0=a(e, v_h) =  a(\epsilon_h, v_h) + a(\zeta_h, v_h)=\int_I s_h v_h dx +\int_I (\zeta_h)_t v_h dx-iA(v_h,\zeta_h), \quad \forall v_h \in V_h^k.
\eeq

Now, we are going to choose three special test functions to extract superconvergence properties \eqref{eqn:zetaxx}-\eqref{eqn:zetaxjump} about $\zeta_h.$ We first prove \eqref{eqn:zetaxx}. 
Due to the invertibility of the coefficient matrix $M,$ there exists a nontrivial function $v_1 \in V_h^k$, such that $\forall j \in \zn, v_1|_{I_j} = \alpha_{j,k-1}L_{j,k-1} + \alpha_{j,k} L_{j,k} + \overline{(\zeta_h)_{xx}}$, $\intj v_1 (\zeta_h)_{xx} dx = \|(\zeta_h)_{xx}\|_{L^2(I_j)}^2$, $\hat v_1 |_{\jp} = 0$ and $\widetilde{(v_1)_x}|_{\jp} = 0$.
Thus $A(v_h, \zeta_h) = \|(\zeta_h)_{xx}\|^2$. Let $v_h=v_1,$ then  \eqref{eqa1} becomes $$0=\int_I s_h v_1 dx +\int_I (\zeta_h)_t v_1 dx-i \|(\zeta_h)_{xx}\|^2.$$
Hence $\|(\zeta_h)_{xx}\|^2  \leq  \|s_h + (\zeta_h)_t\| \cdot \|v_1\|.$ In order to show the estimates for $\|(\zeta_h)_{xx}\|,$ it remains to estimate $\|v_1\|.$


When the assumption A1 holds, the definition of $v_1$ yields the following local system for each pair of $\alpha_{j,k-1}$ and $\alpha_{j,k}$,
$$
(A_j + B_j) \bmat \alpha_{j,k-1}  \\ \alpha_{j,k} \emat = -G \bmat \overline {(\zeta_h)_{xx}^-} \\ \overline{(\zeta_h)_{xxx}^-} \emat \bigg |_{\jp}  - H \bmat \overline {(\zeta_h)_{xx}^+} \\ \overline{(\zeta_h)_{xxx}^+} \emat \bigg |_{\jm}, \quad \forall j \in \zn.
$$
By simple algebra
\beq
\label{eqn:alphas1} 
\bmat \alpha_{j,k-1}  \\ \alpha_{j,k} \emat  = -(A_j + B_j)^{-1}G \bmat 1 & 0 \\ 0 & \frac{1}{h_j} \emat \bmat \overline {(\zeta_h)_{xx}^-} \\ h_j  \overline{(\zeta_h)_{xxx}^-} \emat \bigg |_{\jp} - (A_j + B_j)^{-1}H \bmat 1 & 0 \\ 0 & \frac{1}{h_j} \emat \bmat \overline {(\zeta_h)_{xx}^+} \\ h_j \overline{(\zeta_h)_{xxx}^+} \emat \bigg |_{\jm},
\eeq
By orthogonality of Legendre polynomials, it follows that 
\[
	\begin{split}
\|v_1\|^2_{L^2(I_j)} & = |\alpha_{j,k-1}|^2 \intj L_{j,k-1}^2 dx + |\alpha_{j,k}|^2 \intj L_{j,k}^2 dx + \|(\zeta_h)_{xx}\|_{L^2(I_j)}^2\\
					 & \leq C(  h_j  \| (\zeta_h)_{xx} \|_{L^2(\partial I_j)}^2 + h_j^3  \| (\zeta_h)_{xxx} \|_{L^2(\partial I_j)}^2 + \| (\zeta_h)_{xx} \|_{L^2(I_j)}^2) \leq C \| (\zeta_h)_{xx} \|_{L^2(I_j)}^2,
	\end{split}
\]
where Lemma \ref{lem:mjm}, trace inequalities and inverse inequalities are used in above inequality. Therefore, \eqref{eqn:zetaxx} is proven when assumption A1 is satisfied.

Similarly, we define $v_2 \in V_h^k$, such that $\forall j \in \zn, v_2|_{I_j} = \alpha_{j,k-1}L_{j,k-1} + \alpha_{j,k} L_{j,k}$, $\intj v_2 (\zeta_h)_{xx} dx = 0$, $\hat v_2 |_{\jp} = 0$ and $\widetilde{(v_2)_x}|_{\jp} = \overline{[\zeta_h]}|_{\jp}$. Thus $A(v_h, \zeta_h) = -\sumj |[\zeta_h]|_{j+\ot}^2$. When assumption A1 is satisfied, this definition yields the following local system for each pair of $\alpha_{j,k-1}$ and $\alpha_{j,k}$,
$$
(A_j + B_j) \bmat \alpha_{j,k-1}  \\ \alpha_{j,k} \emat = G \bmat 0 \\ \overline{[\zeta_h]}  \emat \bigg |_{\jp}  + H \bmat 0 \\ \overline{[\zeta_h]}  \emat \bigg |_{\jm}, \quad \forall j \in \zn.
$$
By same algebra as above, we have
$$\bmat \alpha_{j,k-1}  \\ \alpha_{j,k} \emat  = (A_j + B_j)^{-1}G \bmat 1 & 0 \\ 0 & \frac{1}{h_j} \emat \bmat 0 \\ h_j \overline{[\zeta_h]}  \emat \bigg |_{\jp}  + (A_j + B_j)^{-1} H \bmat 1 & 0 \\ 0 & \frac{1}{h_j} \emat \bmat 0 \\ h_j \overline{[\zeta_h]}  \emat \bigg |_{\jm}. $$
By Lemma \ref{lem:mjm}, it follows directly that 
$$
\|v_2\|^2_{L^2(I_j)} \leq C h^3_j (  |{[\zeta_h]}|^2_{\jp} + |{[\zeta_h]}|^2_{\jm}).
$$
Plug $v_2$ in \eqref{eqa1}, we obtain
$$
\sumj |[\zeta_h]|_{\jp}^2 = i \int_I s_h v_2 dx +i \int_I (\zeta_h)_t v_2 dx \leq \|s_h + (\zeta_h)_t\| \|v_2\|.
$$
Therefore, \eqref{eqn:zetajump} is proven when assumption A1 is satisfied.

Finally, we can also choose  $v_3 \in V_h^k$, such that  $\forall j \in \zn, v_3|_{I_j} = \alpha_{j,k-1}L_{j,k-1} + \alpha_{j,k} L_{j,k}$ such that $\intj v_3 (\zeta_h)_{xx} dx = 0$, $\hat v_3 |_{\jp} = \overline{[(\zeta_h)_x]}|_{\jp}$ and $\widetilde{(v_3)_x}|_{\jp} = 0$. Thus $A(v_h, \zeta_h) = \sumj |[(\zeta_h)_x]|_{j+\ot}^2$. Follow the same lines as the estimates for $v_2$, we end up with the estimates
$$
\|v_3 \|_{L^2(I_j)}^2 \leq C h_j (  |{[(\zeta_h)_x]}|^2_{\jp} + |{[(\zeta_h)_x]}|^2_{\jm}).
$$
Plug $v_3$ in \eqref{eqa1}, we obtain \eqref{eqn:zetaxjump} when assumption A1 is satisfied.

Under assumption A2, we need to compute $\sumj (|\alpha_{j,k-1}|^2 + |\alpha_{j,k}|^2)$ to estimate $\| v_1 \|^2$. 
The definition of $v_1$ yields the following coupled system
\beq
\label{eqn:vest}
A \bmat \alpha_{j,k-1}  \\ \alpha_{j,k} \emat + B \bmat \alpha_{j+1,k-1}  \\ \alpha_{j+1,k} \emat = -G \bmat \overline {(\zeta_h)_{xx}^-} \\ \overline{(\zeta_h)_{xxx}^-} \emat \bigg |_{\jp}  - H \bmat \overline {(\zeta_h)_{xx}^+} \\ \overline{(\zeta_h)_{xxx}^+} \emat \bigg |_{\jp}, \quad j \in \mathbb{Z}_N.
\eeq
Write it in matrix form
\[
M \bs \alpha = \bs b, \quad \bs \alpha = [ \bs \alpha_1, \cdots , \bs \alpha_N]^T, 
\]
where $M$ is defined in \eqref{eqn:mdef} and 
\[
\bs  \alpha_j = [\alpha_{j,k-1}, \alpha_{j, k}], \bs b = [ \bs b_1, \cdots, \bs b_N]^T, \bs b_j =  -G \bmat \overline {(\zeta_h)_{xx}^-} \\ \overline{(\zeta_h)_{xxx}^-} \emat \bigg |_{\jp}  - H \bmat \overline {(\zeta_h)_{xx}^+} \\ \overline{(\zeta_h)_{xxx}^+} \emat \bigg |_{\jp}.
\]

Multiply $A^{-1}$ from the left in \eqref{eqn:vest}, we get an equivalent system
\[
M' \bs \alpha = \bs b',  \quad M' = circ(I_2, A^{-1}B, 0_2, \cdots, 0_2), \bs b' = [ \bs b_1', \cdots, \bs b_N']^T, \bs b_j' = A^{-1} \bs b_j,
\]
and $(M')^{-1} = circ(r_0, \cdots, r_{N-1})$. By Theorem 5.6.4 in \cite{davis2012circulant} and similar to the proof in Lemma 3.1 in \cite{cao2017superconvergenceddg},
\[
M' = (\mathcal F_{N}^* \otimes I_2) \bs \Omega (\mathcal F_{N} \otimes I_2),
\]
where $\mathcal F_{N}$ is the $N \times N$ discrete Fourier transform matrix defined by $(\mathcal F_N)_{ij} = \frac{1}{\sqrt N} \overline \omega^{(i-1)(j-1)}, \omega = e^{i\frac{2\pi}{N}}.$ $\mathcal F_{N}$ is symmetric and unitary and
\[
  \quad \bs \Omega = \text{diag}( I_2 + A^{-1}B, I_2 + \omega A^{-1}B , \cdots, I_2 + \omega^{N-1} A^{-1}B ).
\]
The assumption $\frac{\abs \g}{\abs \la} > 1$ in A2 ensures that the eigenvalues of $Q = -A^{-1}B$ (see (52) in \cite{2018arXiv180105875C}) are not $1$, thus $I_2 + \omega^j A^{-1}B, \forall j,$ is nonsingular and $\bs \Omega$ is invertible. Then
\beq
\label{eqn:rhominv}
| \rho( (M')^{-1}) | = \| (M')^{-1} \|_2 \leq \| \mathcal F_{N}^* \otimes I_2 \|_2 \| \bs \Omega \|_2 \| \mathcal F_{N} \otimes I_2 \|_2 \leq C.
\eeq

Therefore,
\[
\sumj (|\alpha_{j,k-1}|^2 + |\alpha_{j,k}|^2) = \bs \alpha^T \bs \alpha = (\bs b')^T (M')^{-T} (M')^{-1} (\bs b')^T \leq \| (M')^{-1} \|_2^2 \| \bs b'\|_2^2 \leq C \sumj \| \bs b_j'\|_2^2.
\]

Since $A^{-1}G \bmat 1 & 0 \\ 0 & \frac{1}{h} \emat, A^{-1}H  \bmat 1 & 0 \\ 0 & \frac{1}{h} \emat$ are constant matrices, we have
\[
\begin{split}
\| \bs b_j'\|_2^2 & \leq C \left ( \left \| \bmat \overline {(\zeta_h)_{xx}^-} \\ h \overline{(\zeta_h)_{xxx}^-} \emat \bigg |_{\jp} \right \|_2 + \left \| \bmat \overline {(\zeta_h)_{xx}^+} \\ h \overline{(\zeta_h)_{xxx}^+} \emat \bigg |_{\jp} \right \|_2 	\right )			\\
& \leq C(  \| (\zeta_h)_{xx} \|_{L^2(\partial I_j)}^2 + \| (\zeta_h)_{xx} \|_{L^2(\partial I_{j+1})}^2 +  h^2 \| (\zeta_h)_{xxx} \|_{L^2(\partial I_j)}^2 + h^2 \| (\zeta_h)_{xxx} \|_{L^2(\partial I_{j+1})}^2)									\\
&  \leq C(  \| (\zeta_h)_{xx} \|_{L^2(\partial I_j)}^2 + \| (\zeta_h)_{xx} \|_{L^2(\partial I_{j+1})}^2),
\end{split}
\]
where inverse inequality is used to obtain the last inequality. Finally, we obtain the estimate
\[
\begin{split}
\| v_1 \| ^2 & = \sumj |\alpha_{j,k-1}|^2 \| L_{j,k-1} \|_{L^2 (I_j)}^2 + \sumj |\alpha_{j,k}|^2 \| L_{j,k} \|_{L^2 (I_j)}^2 + \| (\zeta_h)_{xx} \|^2						\\
& \leq \| (\zeta_h)_{xx} \|^2  +  Ch \sumj (|\alpha_{j,k-1}|^2 + |\alpha_{j,k}|^2) 			\\
& \leq \| (\zeta_h)_{xx} \|^2  +  Ch \sumj (  \| (\zeta_h)_{xx} \|_{L^2(\partial I_j)}^2 + \| (\zeta_h)_{xx} \|_{L^2(\partial I_{j+1})}^2) 											\\
& \leq \| (\zeta_h)_{xx} \|^2  + C h \| (\zeta_h)_{xx} \|_{ L^2 (\partial \mathcal I_N)}^2 \leq C \| (\zeta_h)_{xx} \|^2,
\end{split}
\]
where inverse inequality is used to obtain the last inequality. Then the estimates for \eqref{eqn:zetaxx} hold true. \eqref{eqn:zetajump} and \eqref{eqn:zetaxjump} can be proven by the same procedure when assumption A2 is satisfied, and the steps are omitted for brevity.
\end{proof}

 \begin{rem}
	When assumption A3 is satisfied, the eigenvalues of $Q$ are two
	complex number with magnitude 1, then a constant bound for
	$\rho((M')^{-1})$ as in \eqref{eqn:rhominv} is not possible. Therefore, we cannot obtain similar results for
assumption A3.  
\end{rem}

\subsection{Proof for Lemma \ref{lem:west}}
\label{apdx:westproof}
\begin{proof}
	Since $w_q \in V_h^k$, we have
	\beq
		\label{eqn:wqgeneral}
		w_q |_{I_j} = \sum_{m=0}^k c_{j,m}^q L_{j,m}.
	\eeq
	Let $v_h = D^{-2} L_{j,m}, m \leq k - 2$ in \eqref{eqn:wq1}, we obtain
	\begin{equation}
	\label{eqn:crel}
	c_{j,m}^q = -i \frac{2m+1}{h_j} \frac{h_j^2}{4} \int_{I_j} \partial_t w_{q-1} D^{-2} L_{j,m} dx.
	\end{equation}
	
	Since $D^{-2} L_{j,m} \in P_c^{m+2}(I_j)$, by the property $u - \pst u \perp V_h^{k-2}$ in the $L^2$ inner product sense, we have
	\beq
	\label{eqn:cjmq}
	c_{j,m}^1 = 
	\begin{cases}
	-i\frac{2m+1}{h_j} \frac{h_j^2}{4}  \int_{I_j} \partial_t (u - \pst u) D^{-2} L_{j,m} dx = 0, & m \leq k - 4,	\\
	-i\frac{2m+1}{h_j} \frac{h_j^2}{4}  \int_{I_j} \partial_t ((u_{j,k-1} - \acute u_{j,k-1}) L_{j,k-1} + (u_{j,k} - \acute u_{j,k}) L_{j,k}) D^{-2} L_{j,m} dx, & m = k-3, k-2.	\\
	\end{cases}
	\eeq
	%
	By induction using \eqref{eqn:wqgeneral}, \eqref{eqn:crel}, \eqref{eqn:cjmq}, for $0 \leq m \leq k - 2 - 2q$, $c_{j,m}^q = 0.$
	
	Furthermore, the first nonzero coefficient can be written in a simpler form related to 
	$u_{j,k-1}$ by induction.
	
	When $q=1$, we compute $c_{j, k-3}^1$ by \eqref{eqn:cjmq} and the definition of $w_0$. That is
	\[
	\begin{split}
	c_{j, k-3}^1 	& = -i \frac{2(k-3)+1}{h_j} \Big ( \frac{h_j}{2} \Big )^2 \partial_t (u_{j,k-1} - \acute u_{j,k-1}) \intj D^{-2} L_{j,k-3} L_{j, k-1} dx 	\\
				& = C h_j^2 \partial_t (u_{j,k-1} - \acute u_{j,k-1}).
	\end{split}
	\]
	
	Suppose $c_{k+1-2q}^{q-1} = Ch_j^{2q-2}\partial_t^{q-1} (u_{j,k-1} - \acute u_{j,k-1})$, then 
	\[
	\begin{split}
	c_{j, k-1-2q}^q 	& = -i \frac{2(k-1-2q)+1}{h_j} \Big ( \frac{h_j}{2} \Big )^2 \partial_t c_{j,k+1-2q}^{q-1} \intj D^{-2} L_{j,k-1-2q} L_{j, k+1-2q} dx 	\\
			& = C {h_j}^{2q} \partial_t^{q} (u_{j,k-1} - \acute u_{j,k-1}).
	\end{split}
	\]
	The induction is completed and \eqref{eqn:wqexp} is proven when $r=0$.

	Next, we begin estimating the coefficient $c_{j, m}^q.$ By Holder's inequality and \eqref{eqn:crel}, we have the estimates for $c_{j, m}^q, k-1-2q \leq m \leq k-2$,
	\[
	\abs{c_{j,m}^q} \leq C h^{2- \frac{1}{s}} \| \partial_t w_{q-1} \|_{L^s(I_j)}.
	\]
	
	To estimate the coefficients $c_{j, k-1}^q, c_{j, k}^q$, we need to discuss it by cases. If assumption A1 is satisfied, meaning \eqref{eqn:wq2} and \eqref{eqn:wq3} can be decoupled and therefore $w_q$ is locally defined by \eqref{eqn:wqdef}.
	By \eqref{eqn:wqvec} and following the same algebra of solving the $k$-th and $(k+1)$-th coefficients in \eqref{eqn:pdagflux2},
	\[
	\bmat
	 c_{j,k-1}^q	\\
	 c_{j,k}^q
	\emat
	 = -\sum_{m=0}^{k-2} \mathcal{M}_{j, m}  c_{j,m}^q.
	\]
	
	By \eqref{eqn:mjm}, for all $j \in \zn$,
	\[
	\begin{split}
	\abs{c_{j,k-1}^q}^2 + \abs{c_{j,k}^q}^2 & \leq C \sum_{m=k-2q-3}^{k-2} \abs{c_{j,m}^q}^2 \leq C h^3 \| \partial_t w_{q-1} \|^2_{L^2(I_j)}.	\\
	\max (\abs{c_{j,k-1}^q} , \abs{c_{j,k}^q}) & \leq C \max_{k-2q-3 \leq m \leq k-2} \abs{c_{j,m}^q} \leq C h^2 \| \partial_t w_{q-1} \|_{L^\infty(I_j)}.
	\end{split}
	\]
	
	If one of assumption A2/A3 is satisfied, \eqref{eqn:wqvec} defines a coupled system. From the same lines for obtaining \eqref{eqn:pstcoefa23} in Appendix \ref{apdx:pstproof}, the solution for $ c_{j,k-1}^q,  c_{j,k}^q$ is
	\beq
	\label{eqn:cjk-1k}
	\begin{split}
	\bmat
	 c_{j,k-1}^q	\\
	 c_{j,k}^q
	\emat
	& = -\sum_{m=k-1-2q}^{k-2} \sum_{l=0}^{N-1} r_l A^{-1} (G  L_{m}^- c_{j+l,m}^q + H L_{m}^+   c_{j+l+1,m}^q )		\\
	& = - \sum_{m=k-1-2q}^{k-2}\Big ( c_{j,m}^q V_{1,m} + \suml c_{j+l,m}^q r_l V_{2,m}\Big ).
	\end{split}
	\eeq
	
	Under assumption A2, we can estimate $c_{j,m}^q, m = k - 1, k,$ using \eqref{eqn:minv1}, that is
	\[
	\left \| \bmat
	 c_{j,k-1}^q	\\
	 c_{j,k}^q
	\emat \right \|_\infty
	 \leq C (1+ \suml \|r_l\|_\infty ) \max |c_{j+l, m}^q| 	\leq C h^2 \| \partial_t w_{q-1} \|_{L^\infty(\mathcal I_N)}.
	\]
	
	Under assumption A3, $\suml \|r_l\|_\infty$ is unbounded. Thus we use Fourier analysis to bound the coefficients utilizing the smoothness and periodicity by similar idea in \cite{2018arXiv180105875C}.
	In the rest of the proof, we make use of two operators $\sumla$ and $\sumnj$, which are defined in \eqref{eqn:sumladef1} and \eqref{eqn:sumladef2}.
	
	When $\frac{|\g|}{|\la|} < 1$, $Q=-A^{-1}B$ has two imaginary eigenvalues $\lo, \lt$ with $|\lo | = |\lt | =1$.
	By (59) of \cite{2018arXiv180105875C}, we have $r_l= \frac{\lo^l}{1-\lo^N} Q_1 + \frac{\lt^l}{1-\lt^N} (I_2 - Q_1)$, where $Q_1$ is a constant matrix independent of $h$, and defined in (60) and (61) in \cite{2018arXiv180105875C}. We perform more detailed computation of the coefficients. 
	In \eqref{eqn:cjmq}, plug in \eqref{eqn:pstcoefa23}, for $m = k-3, k-2$, when $u_t \in W^{k+2+n,\infty}(I)$,
	\[
	\begin{split}
	c_{j,m}^1 	& = i \frac{2m+1}{h} \frac{h^2}{4}  \intj [L_{j,k-1}, L_{j,k}] \partial_t \sum_{p=k+1}^\infty  \Big (  u_{j,p} V_{1,p} +  \suml  u_{j+l,p} r_l V_{2,p} \Big ) D^{-2} L_{j,m} dx			\\
			& = i \frac{2m+1}{2} \frac{h^2}{4} \sum_{p = k + 1}^\infty \partial_t \Big (u_{j,p} F_{p,m}^1 + \suml u_{j+l,p}r_l F_{p,m}^2 \Big )	\\
			& = i \frac{2m+1}{8}h^2 \sum_{p=k+1}^\infty \Big ( \sum_{s=0}^n \mu_s h^{k+1+s}u_t^{(k+1+s)}(x_{\jm}) F_{p,m}^1 				\\
			& + \suml \big (\frac{\lo^l}{1-\lo^N} Q_1 + \frac{\lt^l}{1-\lt^N} (I_2 - Q_1) \big ) \sum_{s=0}^n \mu_s h^{k+1+s}u_t^{(k+1+s)}(x_{j+l-\ot}) F_{p,m}^2  + O(h^{k+n+1}|u_t|_{W^{k+2+n,\infty}(I)}) \Big )		\\
			& = i \frac{2m+1}{8}h^2 \sum_{p=k+1}^\infty  \sum_{s=0}^n \mu_s h^{k+1+s} \big (u_t^{(k+1+s)}(x_{\jm}) F_{p,m}^1 	\\
			& + (Q_1 \sumla_{\lo} + (I_2-Q_1) \sumla_{\lt}) u_t^{(k+1+s)} (x_{\jm})F^2_{p,m} \big) + O(h^{k+3+n}|u_t|_{W^{k+2+n,\infty}(I)}),
	\end{split}
	\]
	where $F_{p,m}^\nu = \frac{2}{h} \intj [L_{j,k-1}, L_{j,k}] V_{\nu,p} D^{-2} L_{j,m} dx, \nu = 1, 2,$ are constants independent of $h$ and \eqref{eqn:ujmtaylor} is used in the third equality.
	
	Plug the formula above into \eqref{eqn:cjk-1k}, by similar computation, we have
	\[
	\begin{split}
	\bmat
	 c_{j,k-1}^1	\\
	 c_{j,k}^1
	\emat
	& = -i\frac{2m+1}{8}h^2 \sum_{m=k-3}^{k-2} \sum_{p=k+1}^\infty  \sum_{s=0}^n \mu_s h^{k+1+s} \big (u_t^{(k+1+s)} (x_{\jm}) F_{p,m}^1 V_{1,m} \\
	& + (Q_1 \sumla_{\lo} + (I_2-Q_1) \sumla_{\lt}) u_t^{(k+1+s)}(x_{\jm}) (F_{p,m}^2 V_{1,m} + F_{p,m}^1 V_{2, m}) \\
			& + (Q_1 \sumla_{\lo} + (I_2-Q_1) \sumla_{\lt})^2 u_t^{(k+1+s)} (x_{\jm})F^2_{p,m} V_{2,m} \big) + O(h^{k+2+n}|u_t|_{W^{k+2+n,\infty}(I)}).
	\end{split}
	\]
	By \eqref{eqn:sumlaest}, we have
	\[
	(Q_1 \sumla_{\lo} + (I_2-Q_1) \sumla_{\lt})^\nu u_t^{(k+1+s)} (x_{\jm}) \leq C |u_t|_{W^{k+2+s+\nu,1}(I)} \leq C |u|_{W^{k+4+s+\nu,1}(I)}.
	\]
	Therefore,
	\[
	|c_{j, m}^1 | \leq C_2 h^{k+3}, \quad m = k-3,k-2, \ \text{ and } |c_{j, m}^1 | \leq C_3 h^{k+3}, \quad m = k-1,k.
	\]
	
	By induction and similar computation, we can obtain the formula for $c_{j, m}^q$. For brevity, we omit the computation and directly show the estimates
	\[
	|c_{j, m}^q | \leq C_{3q} h^{k+1+2q}, \quad k-1-2q \leq m \leq k.
	\]
	
	
	When $\frac{|\g|}{|\la|} =1$, $Q=-A^{-1}B$ has two repeated eigenvalues. By (71) of \cite{2018arXiv180105875C},
	 we have $r_l = \frac{(-1)^l}{2} I_2 + (-1)^l \frac{-N+2l}{4\g} Q_2$, where $Q_2/ \g$ is a constant matrix, 
	 then by \eqref{eqn:pstcoefa23} and \eqref{eqn:ujmtaylor}.
	 For $m = k-3, k-2$, when $u_t \in W^{k+2+n,\infty}(I)$, we compute $c_{j, m}^1$ by the same procedure as previous case and obtain
	\[
	\begin{split}
	c_{j, m}^1 & = i \frac{2m+1}{8}h^2 \sum_{p=k+1}^\infty  \sum_{s=0}^n \mu_s h^{k+1+s} \big (u_t^{(k+1+s)}(x_{\jm}) F_{p,m}^1 \\
			& + \ot (\sumla_{-1} + \frac{Q_2}{\g} \sumnj ) u_t^{(k+1+s)} (x_{\jm}) F^2_{p,m} \big) + O(h^{k+3+n}|u_t|_{W^{k+2+n,\infty}(I)}).
	\end{split}
	\]
	
	Plug formula above into \eqref{eqn:cjk-1k}, we have
	\[
	\begin{split}
	\bmat
	 c_{j,k-1}^1	\\
	 c_{j,k}^1
	\emat
	& = -i \frac{2m+1}{8}h^2 \sum_{m=k-3}^{k-2}\sum_{p=k+1}^\infty  \sum_{s=0}^n \mu_s h^{k+1+s} \big ( u_t^{(k+1+s)}(x_{\jm}) F_{p,m}^1 V_{1,m} \\
			& + \ot (\sumla_{-1} + \frac{Q_2}{\g} \sumnj ) u_t^{(k+1+s)}(x_{\jm}) (F_{p,m}^2 V_{1,m} + F_{p,m}^1 V_{2, m}) \\
			& + \frac{1}{4}(\sumla_{-1} + \frac{Q_2}{\g} \sumnj )^2 u_t^{(k+1+s)} (x_{\jm})F^2_{p,m} V_{2,m} \big) + O(h^{k+2+n}|u_t|_{W^{k+2+n,\infty}(I)}).
	\end{split}
	\]
	
	By \eqref{eqn:sumlaest}, we have
	\[
	(\sumla_{-1} + \frac{Q_2}{\g} \sumnj )^\nu u_t^{(k+1+s)} (x_{\jm}) \leq C |u_t|_{W^{k+2+s+2\nu,1}(I)} \leq C |u|_{W^{k+4+s+2\nu,1}(I)}
	\]
	and
	\[
	|c_{j, m}^1 | \leq C_2 h^{k+3}, \quad m = k-3,k-2, \ \text{ and } |c_{j, m}^1 | \leq C_4 h^{k+3}, \quad m = k-1,k.
	\]
	
	By induction and similar computation, we can obtain the formula for $c_{j, m}^q$. For brevity, we omit the computation and directly show the estimates
	\[
	|c_{j, m}^q | \leq C_{4q} h^{k+1+2q}, \quad k-1-2q \leq m \leq k.
	\]
	
	
	All the analysis above works when we change definition of $w_q$ to $\partial_t^r w_q$ (and change $(w_{q-1})_t$ to $\partial_t^{r+1} w_q$ accordingly) in \eqref{eqn:wqdef}.
	Summarize the estimates for $c_{j, m}^q$ under all three assumptions, for $1 \leq q \leq \floor{\frac{k-1}{2}}$, we have
	\[
		|\partial_t^r c_{j,m}^q| \leq C_{2r,q} h^{k+1+2q}, \quad \| \partial_t^r w_q \| \leq C ( \sumj \sum_{m= k-2q -1}^k |\partial_t^r c_{j,m}^q|^2 h_j)^{\ot} \leq C_{2r,q}h^{k+1+2q}.
	\]
	Then \eqref{eqn:wqexp}, \eqref{eqn:west2} is proven. And \eqref{eqn:wqerreq} is a direct result of above estimate and \eqref{eqn:aui}.
	\end{proof}

\end{document}  

\end{document}